\documentclass[12pt,a4paper]{article}
\usepackage{latexsym,amssymb,amsmath,mathrsfs,amsthm}
\usepackage{upgreek}
\usepackage{color}
\usepackage{marvosym}
\usepackage{wasysym}
\usepackage{sectsty}
\usepackage{bm}
\usepackage[T1]{fontenc}
\usepackage[anchorcolor=blue,%
 bookmarks=true,%
 bookmarksnumbered=true,%
]{hyperref}
\usepackage{pst-all}
\usepackage{graphicx}
\allsectionsfont{\normalsize\sc\center}

\newtheorem{thm}{Theorem}[section]
\newtheorem{prop}[thm]{Proposition}

\newtheorem{lem}[thm]{Lemma}

\newtheorem{defn}[thm]{Definition}
\newtheorem{remark}[thm]{Remark}
\newtheorem{example}[thm]{Example}
\newtheorem{assumption}{Assumption}

\makeatletter \@addtoreset{equation}{section} \makeatother

\renewcommand{\P}{\mathbb{P}}
\newcommand{\E}{\mathbb{E}}

\newcommand{\R}{\mathbb{R}}

\newcommand{\N}{\mathbb{N}}

\newcommand{\e}{\mathrm{e}}

\renewcommand{\d}{\mathrm{d}}
\newcommand{\m}{\mathfrak{m} }

\newcommand{\s}{\mathfrak{s}}

\newcommand{\1}{{\bf 1}}
\newcommand{\eps}{{\varepsilon}}

\newcommand{\wh}{\widehat}
\renewcommand{\wh}{\widehat}

\title{\large\bf The Littlewood-Paley-Stein inequality for 
Dirichlet space tamed by 
distributional curvature lower bounds}
\author{
Syota Esaki\footnote{Syota Esaki ({\sf sesaki@fukuoka-u.ac.jp}) 
Department of Applied Mathematics, Fukuoka University,
Fukuoka 814-0180, Japan. % ({\sf sesaki@fukuoka-u.ac.jp}). 
Supported in part by JSPS Grant-in-Aid for Scientific Research (C) (No. 23K03158) and 
fund (No.~215001) from the Central Research Institute of Fukuoka University. \newline 
Current address: Mathematical Sciences Program, Department of Science and Technology, 
Faculty of Science and Technology, 
Oita University, 700 Dannoharu, Oita City, Oita Pref., 870-1192, JAPAN 
({\sf sesaki@oita-u.ac.jp}).
}\ \ \ \
Zi Jian Xu\footnote{Zi Jian Xu ({\sf a535218668@yahoo.co.jp}) 
Department of Applied Mathematics, Fukuoka University,
Fukuoka 814-0180, Japan. % ({\sf a535218668@} {\sf yahoo.co.jp}).
}
\ \ and\ \
Kazuhiro Kuwae\footnote{\Letter\, Kazuhiro Kuwae ({\sf kuwae@fukuoka-u.ac.jp})
Department of Applied Mathematics, Fukuoka University,
Fukuoka 814-0180, Japan. % ({\sf kuwae@fukuoka-u.ac.jp}). 
Supported in part by JSPS Grant-in-Aid for Scientific Research (S) (No. 22H04942) and fund (No.~215001) from the Central Research Institute of Fukuoka University.}
}
\date{}
\begin{document}
\maketitle
\begin{abstract}
The notion of tamed Dirichlet space by distributional lower Ricci curvature bounds was proposed by Erbar-Rigoni-Sturm-Tamanini~\cite{ERST} as the Dirichlet space having a weak form of Bakry-\'Emery curvature lower bounds in distribution sense. 
In this framework,    
we establish the Littlewood-Paley-Stein inequality for $L^p$-functions and $L^p$-boundedness of $q$-Riesz transforms with $q>1$, which partially 
generalizes the result by Kawabi-Miyokawa~\cite{KawabiMiyokawa}. 
\end{abstract}

{\it Keywords}:  Strongly local Dirichlet space, tamed Dirichlet space, Bakry-\'Emery condition, smooth measures of (extended) 
Kato class, smooth measures of Dynkin class, 
Littlewood-Paley-Stein inequality, Riesz transform, Wiener space, path space with Gibbs measure,   RCD-space, 
Riemannian manifold with boundary, configuration space,  
{\it Mathematics Subject Classification (2020)}: Primary 31C25, 60H15, 60J60 ; 
Secondary 30L15, 53C21, 58J35, 35K05, 42B05, 47D08

%%%%%%%%%%%%%%%%%%%%%
%%%%%%%%%%%%%%%%%%%%%
%%%%%%%%%%%%%%%%%%%%%
\section{Statement of Main Theorem}\label{sec:StatementMain}
%%%%%%%%%%%%%%%%%%%%%
\subsection{Framework}\label{subsec:Frame}
Let $(X,\tau)$ be a topological Lusin space, i.e., a continuous injective image of a Polish 
space, endowed with a $\sigma$-finite Borel measure $\m$ on $X$ with full topological support. 
Let $(\mathscr{E},D(\mathscr{E}))$ be a quasi-regular symmetric strongly local Dirichlet space on $L^2(X;\m)$ 
and $(P_t)_{t\geq0}$ the associated symmetric sub-Markovian strongly continuous semigroup on $L^2(X;\m)$ (see \cite[Chapter IV, Definition~3]{MR} for the quasi-regularity and see \cite[Theorem~5.1(i)]{Kw:func} for the strong locality). 
Then there exists an $\m$-symmetric special standard process ${\bf X}=(\Omega, X_t, \P_x)$ 
associated with $(\mathscr{E},D(\mathscr{E}))$, i.e. for $f\in L^2(X;\m)\cap \mathscr{B}(X)$, 
$P_tf(x)=\E_x[f(X_t)]$ $\m$-a.e.~$x\in X$ (see \cite[Chapter IV, Section~3]{MR}). 
Here $\mathscr{B}(X)$ denotes the family of Borel measurable functions on $M$ (the symbol $\mathscr{B}(X)$ is also used for the family of Borel measurable subsets of $M$ in some context). 
Here $\mathscr{B}(X)$ denotes the family of Borel measurable functions on $M$ (the symbol $\mathscr{B}(X)$ is also used for the family of Borel measurable subsets of $M$ in some context). 
An increasing sequence $\{F_n\}$ of closed subsets is called an \emph{$\mathscr{E}$-nest} if 
$\bigcup_{n=1}^{\infty}D(\mathscr{E})_{F_n}$ is $\mathscr{E}_1$-dense in $D(\mathscr{E})$, where 
$D(\mathscr{E})_{F_n}:=\{u\in D(\mathscr{E})\mid u=0\;\m\text{-a.e.~on }X\setminus F_n\}$. 
A subset $N$ is said to be \emph{$\mathscr{E}$-exceptional} or \emph{$\mathscr{E}$-polar} if 
there exists an $\mathscr{E}$-nest $\{F_n\}$ such that $N\subset \bigcap_{n=1}^{\infty}(X\setminus F_n)$. 
It is known that any $\mathscr{E}$-exceptional set is $\m$-negligible and for any $\m$-negligible 
$\mathscr{E}$-quasi-open set is $\mathscr{E}$-exceptional, in particular, for an $\mathscr{E}$-quasi continuous function $u$, $u\geq0$ $\m$-a.e. implies $u\geq0$ $\mathscr{E}$-q.e.   
For a statement $P(x)$ with respect to $x\in X$, we say that $P(x)$ holds $\mathscr{E}$-q.e.~$x\in X$ (simply $P$ holds $\mathscr{E}$-q.e.) if the set $\{x\in X\mid P(x)\text{ holds }\}$ is $\mathscr{E}$-exceptional. 
A subset $G$ of $X$ is called an \emph{$\mathscr{E}$-quasi-open set} if there exists an $\mathscr{E}$-nest $\{F_n\}$ such that $G\cap F_n$ is an open set of $F_n$ with respect to the relative topology on 
$F_n$ for each $n\in\mathbb{N}$. A function $u$ on $X$ is said to be \emph{$\mathscr{E}$-quasi continuous} if there exists an $\mathscr{E}$-nest $\{F_n\}$ such that $u|_{F_n}$ is continuous on $F_n$ for each $n\in\mathbb{N}$.

It is known that for $u,v\in D(\mathscr{E})\cap L^{\infty}(X;\m)$ there exists a unique signed finite Borel 
measure $\mu_{\langle u,v\rangle}$ on $X$ such that 
\begin{align*}
2\int_X\tilde{f}\d\mu_{\langle u,v\rangle}=\mathscr{E}(uf,v)+\mathscr{E}(vf,u)-\mathscr{E}(uv,f)\quad \text{ for }\quad u,v\in D(\mathscr{E})\cap L^{\infty}(X;\m).
\end{align*}
Here $\tilde{f}$ denotes the $\mathscr{E}$-quasi-continuous $\m$-version of $f$ (see \cite[Theorem~2.1.3]{FOT}, \cite[Chapter IV, Proposition~3.3(ii)]{MR}). 
We set $\mu_{\langle f\rangle}:=\mu_{\langle f,f\rangle}$ for $f\in D(\mathscr{E})\cap L^{\infty}(X;\m)$. 
Moreover, 
for $f,g\in D(\mathscr{E})$, there exists a signed finite measure $\mu_{\langle f,g\rangle}$ on $X$ such that 
$\mathscr{E}(f,g)=\mu_{\langle f,g\rangle}(X)$, hence $\mathscr{E}(f,f)=\mu_{\langle f\rangle}(X)$. 
We assume $(\mathscr{E},D(\mathscr{E}))$ admits a carr\'e-du-champ $\Gamma$, i.e. 
$\mu_{\langle f\rangle}\ll\m$ for all $f\in D(\mathscr{E})$ and set $\Gamma(f):=\d\mu_{\langle f\rangle}/\d\m$. 
Then $\mu_{\langle f,g\rangle}\ll\m$ for all $f,g\in D(\mathscr{E})$ and $\Gamma(f,g):=
\d\mu_{\langle f,g\rangle}/{\d\m}\in L^1(X;\m)$ is expressed by $\Gamma(f,g)=\frac14(\Gamma(f+g)-\Gamma(f-g))$ for $f,g\in D(\mathscr{E})$. 

Fix $q\in\{1,2\}$. Let $\kappa$ be a signed smooth measure with its Jordan-Hahn decomposition $\kappa=\kappa^+-\kappa^-$. 
We assume that $\kappa^+$ is smooth measure of Dynkin class ($\kappa^+\in S_D({\bf X})$ in short) 
and $2\kappa^-$ is a smooth measure of extended Kato class ($2\kappa^-\in S_{E\!K}({\bf X})$ in short). 
More precisely, $\nu\in S_D({\bf X})$ (resp.~$\nu\in S_{E\!K}({\bf X})$) if and only if $\nu\in S({\bf X})$ and 
$\m\text{-}\sup_{x\in X}\E_x[A_t^{\nu}]<\infty$ for any/some $t>0$ 
(resp.~$\lim_{t\to0}\m\text{-}\sup_{x\in X}\E_x[A_t^{\nu}]<1$) (see \cite{AM:AF}). 
Here $S({\bf X})$ denotes the family of smooth measures with respect to ${\bf X}$ (see \cite[Chapter VI, Definition~2.3]{MR}, \cite[p.~83]{FOT} for the definition of smooth measures) and $\m$-$\sup_{x\in X}f(x)$ denotes the $\m$-essentially supremum for  a function $f$ on $X$. 
Then, for $q\in\{1,2\}$, the quadratic form 
\begin{align*}
\mathscr{E}^{q\kappa}(f):=\mathscr{E}(f)+\langle q\kappa, \tilde{f}^2\rangle
\end{align*}
with finiteness domain $D(\mathscr{E}^{q\kappa})=D(\mathscr{E})$ is closed, lower semi-bounded, moreover, 
there exists $\alpha_0>0$ and $C>0$ such that 
\begin{align*}
C^{-1}\mathscr{E}_1(f,f)\leq \mathscr{E}^{q\kappa}_{\alpha_0}(f,f)\leq C\mathscr{E}_1(f,f)\quad \text{ for all }\quad f\in D(\mathscr{E}^{q\kappa})=D(\mathscr{E})
\end{align*}
(see \cite[(3.3)]{CFKZ:Pert} and \cite[Assumption of Theorem~1.1]{CFKZ:GenPert}).   
The Feynman-Kac semigroup $(p_t^{q\kappa})_{t\geq0}$ defined by 
\begin{align*}
p_t^{q\kappa}f(x)=\E_x[e^{-qA_t^{\kappa}}f(X_t)],\quad f\in \mathscr{B}_b(X)
\end{align*}
is $\m$-symmetric, i.e. 
\begin{align*}
\int_Xp_t^{q\kappa}f(x)g(x)\m(\d x)=\int_Xf(x)p_t^{q\kappa}g(x)\m(\d x)\quad\text{ for all }\quad f,g\in \mathscr{B}_+(X)
\end{align*} 
and coincides with the strongly continuous semigroup $(P_t^{q\kappa})_{t\geq0}$ on $L^2(X;\m)$ associated with 
$(\mathscr{E}^{q\kappa}, D(\mathscr{E}^{q\kappa}))$ (see \cite[Theorem~1.1]{CFKZ:GenPert}). 
Here $A_t^{\kappa}$ is a continuous additive functional (CAF in short) associated with the signed smooth measure $\kappa$ under Revuz correspondence. 
Under $2\kappa^{-}\in S_{E\!K}({\bf X})$ and $p\in[2,+\infty]$, 
$(p_t^{\kappa})_{t\geq0}$ can be extended 
to be a bounded operator on $L^{p}(X;\m)$ denoted by  $P_t^{\kappa}$ such that 
there exist finite constants $C=C(\kappa)>0, C_{\kappa}\geq0$ depending only on $\kappa^-$ such that 
for every $t\geq0$
\begin{align}
\|P_t^{\kappa}\|_{p,p}\leq Ce^{C_{\kappa}t}.\label{eq:KatoContraction}
\end{align}
%Moreover, under $2\kappa^{-}\in S_{E\!K}({\bf X})$, we can deduce
%\begin{align}
%\|P_t^{\kappa}\|_{p,p}\leq Ce^{C_{\kappa}t}\label{eq:KatoContraction}
%\end{align}
%for $p\in[2,+\infty]$. 
%where we use the convention that $\frac{p}{p-1}:=1$ if $p=\infty$.  
Here $C=1$ and 
$C_{\kappa}\geq0$  
can be taken to be $0$ under $\kappa^-=0$ (cf. \cite[Theorem~2.2]{Sznitzman}). When $\kappa=-R\m$ for a constant $R\in \R$, $C_{\kappa}$ (resp.~$C$) can be taken to be $R\lor 0$ (resp.~$1$). Let $\Delta^{q\kappa}$  be the $L^2$-generator associated with $(\mathscr{E}^{q\kappa},D(\mathscr{E}))$ 
defined by   
\begin{align}
\left\{\begin{array}{ll} D(\Delta^{q\kappa})&:=\{u\in D(\mathscr{E})\mid \text{there exists } w\in L^2(X;\m)\text{ such that }\\
&\hspace{3cm}\mathscr{E}^{q\kappa}(u,v)=-\int_Xwv\,\d\m\quad \text{ for any }\quad v\in D(\mathscr{E})\}, \\ \Delta^{q\kappa} u&:=w\quad\text{ for } w\in L^2(X;\m)\quad\text{specified as above,} \end{array}\right.\label{eq:generatorL2}
\end{align}
called the {\it Schr\"odinger operator} with potential $q\kappa$.  
Formally, $\Delta^{q\kappa}$ can be understood as \lq\lq$\Delta^{q\kappa}=\Delta-q\kappa$\rq\rq, 
where $\Delta$ is the $L^2$-generator associated with $(\mathscr{E},D(\mathscr{E}))$.  
\begin{defn}[{{\bf {\boldmath$q$}-Bakry-\'Emery condition}}]
{\rm 
Suppose that $q\in\{1,2\}$, $\kappa^+\in S_D({\bf X})$, $2\kappa^-\in S_{E\!K}({\bf X})$ and $N\in[1,+\infty]$. We say that $(X,\mathscr{E},\m)$ or simply $X$ satisfies the $q$-Bakry-\'Emery condition, briefly 
${\sf BE}_q(\kappa,N)$, if for every $f\in  D(\Delta)$ with $\Delta f\in D(\mathscr{E})$ 
and every nonnegative $\phi\in D(\Delta^{q\kappa})$ with 
$\Delta^{q\kappa}\phi\in L^{\infty}(X;\m)$ (we further impose $\phi\in L^{\infty}(X;\m)$ for $q=2$), we have 
\begin{align*}
\frac{1}{q}\int_X\Delta^{q\kappa}\phi \Gamma(f)^{\frac{q}{2}}\d\m-\int_X\phi
\Gamma(f)^{\frac{q-2}{2}}
\Gamma(f,\Delta f)\d\m
\geq \frac{1}{N}\int_X\phi\Gamma(f)^{\frac{q-2}{2}}(\Delta f)^2\d\m.
\end{align*}
The latter term is understood as $0$ if $N=\infty$.
}
\end{defn}

\begin{assumption}\label{asmp:Tamed}
{\rm We assume that $X$ satisfies ${\sf BE}_{2}(\kappa,\infty)$ condition for a given signed smooth measure  
$\kappa$ with $\kappa^+\in S_D({\bf X})$ and $2\kappa^-\in S_{E\!K}({\bf X})$. % and $N\in[1,\infty]$.
}
\end{assumption}

Under Assumption~\ref{asmp:Tamed}, we say that $(X,\mathscr{E},\m)$ or simply $X$ is {\it tamed}. 
In fact, under $\kappa^+\in S_D({\bf X})$ and $2\kappa^-\in S_{E\!K}({\bf X})$, the condition ${\sf BE}_2(\kappa,\infty)$ is {\it equivalent} 
to ${\sf BE}_1(\kappa,\infty)$ (see Lemma~\ref{lem:BakryEmeryEquivalence} below), in particular, the heat flow $(P_t)_{t\geq0}$ 
satisfies 
\begin{align}
\Gamma(P_tf)^{1/2}\leq P_t^{\kappa}\Gamma(f)^{1/2}\quad\m\text{-a.e.~for any }f\in D(\mathscr{E})\quad \text{ and }\quad t\geq0\label{eq:gradCont}
\end{align}
(see \cite[Definition~3.3 and Theorem~3.4]{ERST}).
The inequality \eqref{eq:gradCont} plays a crucial role in our paper. 
Note that our condition $\kappa^+\in S_D({\bf X})$, $\kappa^-\in S_{E\!K}({\bf X})$ (resp.~$\kappa^+\in S_D({\bf X})$, $2\kappa^-\in S_{E\!K}({\bf X})$) is stronger than the $1$-moderate (resp.~$2$-moderate) condition treated in \cite{ERST} for the definition of tamed space.  
The $\m$-symmetric Markov process ${\bf X}$ treated in our paper may not be conservative in general. Under Assumption~\ref{asmp:Tamed}, 
a sufficient condition ({\it intrinsic stochastic completeness} called in \cite{ERST}) 
for conservativeness of ${\bf X}$ is discussed in \cite[Section~3.3]{ERST}, in 
particular, under $1\in D(\mathscr{E})$ with $\mathscr{E}(1)=0$ and Assumption~\ref{asmp:Tamed}, we have the conservativeness of ${\bf X}$.

Let us introduce the Littlewood-Paley $G$-functions. To do this, we recall 
the subordination of semigroups. For $t\geq0$, define a probability measure 
$\lambda_t$ on $[0,+\infty[$ whose Laplace transform 
is given by 
\begin{align*}
\int_0^{\infty}e^{-\gamma s}\lambda_t(\d s)=e^{-\sqrt{\gamma}t},\quad \gamma\geq0.
\end{align*}
Then, for $t>0$, $\lambda_t$ has the following expression
\begin{align*}
\lambda_t(\d s):=\frac{t}{2\sqrt{\pi}}e^{-t^2/4s}s^{-3/2}\d s.
\end{align*}
For $\alpha\geq C_{\kappa}$, we define the subordination $(Q_t^{(\alpha),\kappa})_{t\geq0}$ of $(P_t^{\kappa})_{t\geq0}$ by 
\begin{align*}
Q_t^{(\alpha),\kappa}f:=\int_0^{\infty}e^{-\alpha s}P_s^{\kappa}f\,\lambda_t(\d s),\quad f\in  L^p(X;\m).
\end{align*}
When $\kappa=0$, we write $Q_t^{(\alpha)}$ instead of $Q_t^{(\alpha),0}$. Then we can easily see that for $p\in[2,+\infty]$  
\begin{align}
\|Q_t^{(\alpha),\kappa}f\|_{L^p(X;\m)}&\leq \int_0^{\infty}e^{-\alpha s}\|P_s^{\kappa}f\|_{L^p(X;\m)}\lambda_t(\d s)\notag\\
 &\leq C\left(\int_0^{\infty}e^{-(\alpha-C_{\kappa}) s}\lambda_t(\d s) \right)\|f\|_{L^p(X;\m)}\label{eq:Contra}\\
 &=Ce^{-\sqrt{\alpha-C_{\kappa}}t}\|f\|_{L^p(X;\m)}.\notag
\end{align} 
The infinitesimal generator of 
$(Q_t^{(\alpha),\kappa})_{t\geq0}$ is denoted by $-\sqrt{\alpha-\Delta_p^{\kappa}}$. We may omit the subscript $p$ for simplicity. 

For $f\in L^2(X;\m)\cap L^p(X;\m)$ and $\alpha\geq C_{\kappa}$, we define the Littlewood-Paley's $G$-functions by 
\begin{align*}
{g_{\stackrel{\rightarrow}{f}}}^{\kappa}(x,t)&:=\left|\frac{\partial}{\partial t}(Q_t^{(\alpha),\kappa}f)(x) \right|,\qquad\qquad\qquad {G_{\stackrel{\rightarrow}{f}}}^{\!\!\kappa}(x):=\left(\int_0^{\infty}t{g_{\stackrel{\rightarrow}{f}}}^{\!\!\kappa}(x,t)^2\d t  \right)^{1/2},\\
{g_f^{\uparrow}}^{\stackrel{}{\kappa}}
(x,t)&:=\left(\Gamma(Q_t^{(\alpha),\kappa}f) \right)^{1/2}(x),\qquad\qquad\quad {G_f^{\uparrow}}^{\kappa}(x):=\left(\int_0^{\infty}t{g_f^{\uparrow}}^{\kappa}(x,t)^2\d t  \right)^{1/2},\\
g_f^{\kappa}(x,t)&:=\sqrt{{g_{\stackrel{\rightarrow}{f}}}^{\kappa}(x,t)^2+
{g_f^{\uparrow}}^{\kappa}
(x,t)^2}, \qquad\quad G_f^{\;\kappa}(x):=\left(\int_0^{\infty}tg_f^{\kappa}(x,t)^2\d t \right)^{1/2}.
\end{align*}
For notational simplicity, we omit $\kappa$ in $G$-functions when we replace $P_t^{\kappa}$ with $P_t$, e.g. 
 we write ${g_f^{\rightarrow}}(x,t)$ (resp.~${G_f^{\rightarrow}}(x)$) instead of 
${g_f^{\rightarrow}}^{0}(x,t)$ (resp.~${G_f^{\rightarrow}}^{0}(x)$) and so on. 

Now we present the Littlewood-Paley-Stein inequality. In what follows, the notation 
$\|u\|_{L^p(X;\m)}\lesssim \|v\|_{L^p(X;\m)}$ stands for 
$\|u\|_{L^p(X;\m)}\leq A(\kappa)\|v\|_{L^p(X;\m)}$, 
where $A(\kappa)$ is a positive constant depending only on $\kappa$ and $p\in[1,+\infty[$.

\begin{thm}\label{thm:main1}
For any $p\in]1,+\infty[$ and $\alpha\geq C_{\kappa}$ and $\alpha>0$, 
the $G$-functions $g_f^{\rightarrow}(\cdot,t)$, $g_f^{\uparrow}(\cdot,t)$, $g_f(\cdot,t)$, 
$G_f^{\rightarrow}$, $G_f^{\uparrow}$ and $G_f$ can be extended to be in $L^p(X;\m)$ for $f\in L^p(X;\m)$
and the following inequalities hold for 
$f\in L^p(X;\m)$
\begin{align}
\|G_f\|_{L^p(X;\m)}&\lesssim\|f\|_{L^p(X;\m)},\label{eq:LittlewoodPaleyStein1}\\
\|f\|_{L^p(X;\m)}&\lesssim\| G_f^{\rightarrow}\|_{L^p(X;\m)}.\label{eq:LittlewoodPaleyStein2}
\end{align}
Moreover, \eqref{eq:LittlewoodPaleyStein1} remains valid for $\alpha>0$ under $p\in]1,2]$. 
Furthermore, for $p\in]1,+\infty[$, \eqref{eq:LittlewoodPaleyStein1} remains valid for $\alpha=0$ under $\kappa^-=0$ {\rm(}hence $C_{\kappa}=0${\rm)}, and \eqref{eq:LittlewoodPaleyStein2} holds for  
 $\alpha=0$, $\kappa^-=0$ and the transience
of $(\mathscr{E},D(\mathscr{E}))$. 
Suppose further that $(\mathscr{E},D(\mathscr{E}))$ is transient or irreducible, $\alpha=0$ and $\kappa^-=0$. 
Then, for any $p\in]1,+\infty[$, 
the following inequalities hold for 
$f\in L^p(X;\m)\cap L^2(X;\m)$: $E_of\in L^p(X;\m)$, 
\begin{align}
\|G_f\|_{L^p(X;\m)}&\lesssim\|f-E_{o}f\|_{L^p(X;\m)},\label{eq:LittlewoodPaleyStein1+}\\
\|f-E_{o}f\|_{L^p(X;\m)}&\lesssim\| G_f^{\rightarrow}\|_{L^p(X;\m)},\label{eq:LittlewoodPaleyStein2+}\\
\|f-E_{o}f\|_{L^p(X;\m)}&\lesssim\| G_{f}^{\uparrow}\|_{L^p(X;\m)}.\label{eq:LittlewoodPaleyStein3+}
\end{align}
Here $E_{o}:=E_0-E_{0-}$ and $(E_{\lambda})_{\lambda\in\R}$ is a real resolution of the identity associated to the $L^2$-generator $\Delta$, {\rm(}in another word $E_o:=E_{\{0\}}$ and $E_{\lambda}:=E_{]-\infty,\lambda]}${\rm)}
 {\rm(}see \cite[Chapter IX, Section~5]{Yosida} for resolution of the identity{\rm)}. 
\end{thm}

In Theorem~\ref{thm:main1}, there is no estimate of $G_f^{\uparrow}$ when $\alpha>0$. 
If we assume the spectral gap for $(\mathscr{E},D(\mathscr{E}))$, i.e., there exists small $\eps>0$ such that 
$E_{]0,\eps[}=0$, then the Green operator $G_0$ 
with index $0$ is a bounded linear operator on $L^2(X;\m)$ (see \cite[(4.42)]{ShigekawaText}), so that we can deduce the same estimate 
\eqref{eq:LittlewoodPaleyStein2} for $G_f^{\uparrow}$ holds by way of the argument in \cite[Subsection~3.2.11]{ShigekawaText} (see the last statement of Subsection~\ref{subsec:LastProof} below).  

\medskip

As an application of Theorem~\ref{thm:main1}, we provide the following: 

\begin{thm}\label{thm:main2}
Let $p\in]1,+\infty[$, $q\in]1,+\infty[$ and $\alpha> C_{\kappa}$. We define 
\begin{align*}
R_{\alpha}^{(q)}(\Delta)f:=\Gamma\left((\alpha-\Delta)^{-\frac{q}{2}}f \right)^{1/2},\qquad f\in L^p(X;\m)\cap L^2(X;\m).
\end{align*}
Then we have the following statements:
\begin{enumerate}
\item[{\rm (1)}] For any $p\in[2,+\infty[$, $R_{\alpha}^{(q)}(\Delta)$ can be extended to a bounded operator on $L^p(X;\m)$. 
The operator norm $\|R_{\alpha}^{(q)}(\Delta)\|_{p,p}$ depends only on $\kappa,p,q$ and $C_{\kappa}$. This implies the inclusion
\begin{align*}
D((I-\Delta_p)^{\frac{q}{2}})\subset H^{1,p}(X;\m),
\end{align*}
where $H^{1,p}(X;\m)$ is the $\|\cdot\|_{H^{1,p}}$-completion of 
\begin{align*}
\mathscr{D}_{1,p}:=\{f\in L^p(X;\m)\cap D(\mathscr{E})\mid \Gamma(f)^{\frac12}\in L^p(X;\m)\}
\end{align*} 
with respect to the 
$(1,p)$-Sobolev norm $\|\cdot\|_{H^{1,p}}$ defined by 
\begin{align*}
\|f\|_{H^{1,p}}:=\left(\|f\|_{L^p(X;\m)}^p+\|\Gamma(f)^{\frac12}\|_{L^p(X;\m)}^p \right)^{\frac{1}{p}}. 
\end{align*}
\item[{\rm (2)}] For any $p\in[2,+\infty[$ and $q\in]1,2[$, there exists a positive constant $C_{p,q}$ such that 
\begin{align}
\|\Gamma(P_tf)^{1/2}\|_{L^p(X;\m)}\leq C_{p,q}\|R_{\alpha}^{(q)}(\Delta)\|_{p,p}(\alpha^{q/2}+t^{-q/2})
\|f\|_{L^p(X;\m)},\quad t>0,\quad f\in L^p(X;\m).
\end{align}
\end{enumerate}
\end{thm}
\begin{remark}
{\rm For each $p\in]1,+\infty[$, the $(1,p)$-Sobolev space $(H^{1,p}(X;\m),\|\cdot\|_{H^{1,p}})$ becomes a uniformly convex (hence reflexive) Banach space provided $(\mathscr{D}_{1,p},\|\cdot\|_{H^{1,p}})$ is closable on $L^p(X;\m)$ (see \cite[Theorem~1.8]{Kw:SobolevSpace}). 
}
\end{remark}
\begin{remark}\label{lem:XdLi}
{\rm The Littelwood-Paley-Stein inequalities for functions and $1$-forms were firstly proved by 
X.-D.~Li~\cite{Xdli:RieszTrans} in the framework of weighted complete Riemannian manifolds 
$(M,g,e^{-\phi}\nu_g)$ with $\phi\in C^2(M)$ if the Bakry-\'Emery Ricci curvature ${\rm Ric}_{\phi}^{\infty}:={\rm Ric}_g+{\rm Hess}\,(\phi)$ 
has a lower bound of Kato class function $K$ (see \cite[Theorem~4.2]
{Xdli:RieszTrans}). Here $\nu_g$ is the Riemannian volume measure. 
For this, he established gradient type estimates (see \cite[(4.8) and (4.9)]{Xdli:RieszTrans}).
}
\end{remark}
\begin{remark}\label{rem:forthcoming}
{\rm In the forthcoming paper \cite{EXK}, we establish Theorem~\ref{thm:main2}(1) for $q=1$ and $p\in]1,+\infty[$, i.e., so-called the boundedness of the Riesz transform $R_{\alpha}(\Delta):=R_{\alpha}^{(1)}(\Delta)$ on $L^p(X;\m)$ for $p\in]1,+\infty[$ based on the vector space calculus for tamed Dirichlet space developed by Braun~\cite{Braun:Tamed2021}.
}
\end{remark}

\section{Test functions}

Under Assumption~\ref{asmp:Tamed}, we now introduce ${\rm Test}(X)$ the set of test functions: 

\begin{defn}\label{def:TestFunc}
{\rm 
Let $(X,\mathscr{E},\m)$ be a tamed space. Let us define the set of {\it test functions} by 
\begin{align*}
{\rm Test}(X):&=\{f\in D(\Delta)\cap L^{\infty}(X;\m)\mid \Gamma(f)\in L^{\infty}(X;\m),\Delta f\in D(\mathscr{E})\}.
\end{align*}
}
\end{defn}

The following lemmas hold. 

\begin{lem}[{{\bf \cite[Proposition~6.8]{ERST}}}]\label{lem:boundedEst}
Under Assumption~\ref{asmp:Tamed}, ${\sf BE}_2(-\kappa^-,+\infty)$ holds. Moreover, for every $f\in L^2(X;\m)\cap L^{\infty}(X;\m)$ and $t>0$, it holds 
\begin{align}
\Gamma(P_tf)\leq\frac{1}{2t}\|P_t^{-2\kappa^-}\|_{\infty,\infty}\cdot\|f\|_{L^{\infty}(X;\m)}^2.\label{eq:BoundedEst}
\end{align}
In particular, for $f\in L^2(X;\m)\cap L^{\infty}(X;\m)$, then $P_tf\in {\rm Test}(X)$. 
\end{lem}

\begin{lem}[{{\bf \cite[Theorems~3.4 and 3.6, Proposition~3.7 and Theorem~6.10]{ERST}}}]\label{lem:BakryEmeryEquivalence}\quad\\
Under $\kappa^+\in S_D({\bf X})$ and $2\kappa^-\in S_{E\!K}({\bf X})$,  
the condition ${\sf BE}_2(\kappa,\infty)$ is equivalent 
to ${\sf BE}_1(\kappa,\infty)$. In particular, we have \eqref{eq:gradCont}. 
\end{lem}

\begin{lem}[{{\bf \cite[Lemma~3.2]{Sav14}}}]\label{lem:algebra}
Under Assumption~\ref{asmp:Tamed}, for every $f\in {\rm Test}(X)$, we have 
$\Gamma(f)\in D(\mathscr{E})\cap L^{\infty}(X;\m)$ and there exists $\mu=\mu^+-\mu^-$ with 
$\mu^{\pm}\in D(\mathscr{E})^*$ such that 
\begin{align}
-\mathscr{E}^{2\kappa}(u,\varphi)=\int_X\tilde{\varphi}\,\d \mu\quad\text{ for all }\quad \varphi\in D(\mathscr{E}).
\end{align}
Moreover, ${\rm Test}(X)$ is an algebra, i.e., for $f,g\in {\rm Test}(X)$, $fg\in {\rm Test}(X)$, 
if further {\boldmath$f$}$\in {\rm Test}(X)^n$, then $\Phi(${\boldmath$f$}$)\in {\rm Test}(X)$ for every smooth 
function $\Phi:\R^n\to\R$ with $\Phi(0)=0$.
\end{lem}

\begin{lem}[{{\bf cf.~\cite[Corollary~6.9]{ERST}}}]\label{lem:DensenessTestFunc}
${\rm Test}(X)\cap L^p(X;\m)$ is dense in $L^p(X;\m)\cap L^2(X;\m)$ both in $L^p$-norm 
and in $L^2$-norm. Moreover, ${\rm Test}(X)$ is dense in $(\mathscr{E},D(\mathscr{E}))$. 
\end{lem}
\begin{proof}[\bf Proof]  
Take $f\in L^p(X;\m)\cap L^2(X;\m)$. We may assume $f\in L^{\infty}(X;\m)$, because 
$f$ is $L^p$(and also $L^2$)-approximated by a sequence $\{f^k\}$ of 
$L^p(X;\m)\cap L^2(X;\m)\cap L^{\infty}(X;\m)$-functions defined 
by $f^k:=(-k)\lor f\land k$. If $f\in L^p(X;\m)\cap L^2(X;\m)\cap L^{\infty}(X;\m)$, 
$P_tf\in {\rm Test}(X)$ by Lemma~\ref{lem:boundedEst} and 
$\{P_tf\}\subset {\rm Test}(X)\cap L^p(X;\m)$ converges to 
$f$ in $L^p$ and in $L^2$ as $t\to0$. If $f\in D(\mathscr{E})$, then $f$ can be approximated by 
$\{P_tf^k\}$ in $(\mathscr{E},D(\mathscr{E}))$. This shows the last statement.
\end{proof} 
\begin{remark}\label{rem:TestFunction}
{\rm 
As proved above, ${\rm Test}(X)$ forms an algebra and dense in $(\mathscr{E},D(\mathscr{E}))$ under 
Assumption~\ref{asmp:Tamed}. However, ${\rm Test}(X)$ is not necessarily a subspace of $C_b(X)$. 
So, we can not apply the same method as in Shigekawa-Yoshida~\cite{ShigekawaYoshida},  Yoshida~\cite{YoshidaNobuo}. 
When the tamed space comes from ${\sf RCD}$-space, the Sobolev-to-Lipschitz property of ${\sf RCD}$-spaces ensures 
${\rm Test}(X)\subset C_b(X)$ (see also ~\cite{HLi}). 
}
\end{remark}

\section{Proof of Theorem~\ref{thm:main1}}
In this section, we prove  Theorem~\ref{thm:main1} by a probabilistic method. 
The original idea is due to Meyer~\cite{Meyer}. The reader is referred to Bakry~\cite{Bakry1}, Shigekawa-Yoshida~\cite{ShigekawaYoshida}, Yoshida~\cite{YoshidaNobuo}. In these papers, they expanded 
$\Delta(Q_t^{(\alpha)}f)^p$, $f\in \mathcal{A}$, by employing the usual functional 
analytic argument in the proof of Littlewood-Paley-Stein inequality. In that calculations, 
they needed to assume the existence of a good core $\mathcal{A}$ like ${\rm Test}(X)$ in this paper. 
Though we have a good core ${\rm Test}(X)$ in the framework of tamed Dirichlet space under 
Bakry-\'Emery condition ${\sf BE}_2(\kappa,\infty)$, we will not follow their method. 
We mimic the method of the proof of Kawabi-Miyokawa~\cite{KawabiMiyokawa} in proving Theorem~\ref{thm:main1}.  
However, our curvature lower bound is not a constant in general, this gives another 
technical difficulty in proving Theorem~\ref{thm:main1}. For this, we should modify the 
method of the proof in  Kawabi-Miyokawa~\cite{KawabiMiyokawa}.  
We prove Theorem~\ref{thm:main1} for $f\in D(\Delta)\cap L^p(X;\m)$ at the beginning. Any $f\in  L^p(X;\m)\cap L^2(X;\m)$ 
with 
$p\in]1,+\infty[$ can be approximated by a sequence in $D(\Delta)\cap L^p(X;\m)$ in $L^p$-norm and in $L^2$-norm.
Then one can conclude the statement of 
Theorem~\ref{thm:main1} for general $f\in  L^p(X;\m)\cap L^2(X;\m)$ (hence $G_f^{\rightarrow}$, $G_f^{\uparrow}$ and $G_f$ can be extended for general $f\in L^p(X;\m)$) in view of the triangle inequality
$|g_{f_1}^{\rightarrow}-g_{f_2}^{\rightarrow}|\leq g_{f_1-f_2}^{\rightarrow}$, 
(resp.~$|g_{f_1}^{\uparrow}-g_{f_2}^{\uparrow}|\leq g_{f_1-f_2}^{\uparrow}$) hence $|G_{f_1}^{\rightarrow}-G_{f_2}^{\rightarrow}| \leq G_{f_1-f_2}^{\rightarrow}$ (resp.~$|G_{f_1}^{\uparrow}-G_{f_2}^{\uparrow}| \leq G_{f_1-f_2}^{\uparrow}$).

\subsection{Preliminaries}\label{subsec:preliminary}

In this subsections, we make some preparations. We have already used the notation $(\P_x)_{x\in X}$ 
to denote the diffusion measure of ${\bf X}$ associated with the Dirichlet form $(\mathscr{E},D(\mathscr{E}))$. 
Let $(B_t,\P_{\stackrel{\rightarrow}{a}})$ be one-dimensional Brownian motion 
starting at $a\in \R$ with the generator $\frac{\partial^2}{\partial a^2}$. 
Denote by $(\mathbb{D},D(\mathbb{D}))$ the Dirichlet form on $L^2(\R)$ associated to 
$(B_t,\P_{\stackrel{\rightarrow}{a}})$.
We set 
$\widehat{X}:=X\times\R$, $\hat{x}:=(x,a)\in \widehat{X}$, 
$\widehat{X}_t:=(X_t,B_t)$, $t\geq0$, $\widehat{\m}:=\m\otimes m$ 
 and $\P_{\hat{x}}:=\P_x\otimes \P_{\stackrel{\rightarrow}{a}}$. 
Then $\widehat{\bf X}:=(\widehat{X}_t,\P_{\hat{x}})$ is an $\widehat{\m}$-symmetric diffusion process on 
$\widehat{X}$ with the (formal) generator $\Delta+\frac{\partial^2}{\partial a^2}$, where $m$ is one-dimensional Lebesgue measure. Note that the life time $\widehat{\zeta}$ of $\widehat{\bf X}$ coincides with the life time $\zeta$ of ${\bf X}$. Since ${\bf X}$ (hence $\widehat{\bf X}$) is not necessarily conservative, throughout this paper, all functions defined on 
$\widehat{X}$ should be regarded as functions on its one point compactification $\widehat{X}_{\widehat{\partial}}:=\widehat{X}\cup\{\widehat{\partial}\}$ vanishing at $\widehat{\partial}$.
We put $\P_{\m}:=\int_X\P_x\m(\d x)$, 
$\P_{\m\otimes\delta_a}:=\int_X\P_{(x,a)}\m(\d x)$ and denote the integration with respect to 
$\P_x$, $\P_{\stackrel{\rightarrow}{a}}$, $\P_{(x,a)}$ and $\P_{\m\otimes\delta_a}$ by 
$\E_x$, $\E_{\stackrel{\rightarrow}{a}}$, $\E_{(x,a)}$ and $\E_{\m\otimes\delta_a}$, respectively. 

We denote the semigroup on $L^p(\widehat{X};\widehat{\m})$ associated with the diffusion process 
$(\widehat{X}_t)_{t\geq0}$ by $(\widehat{P}_t)_{t\geq0}$ and its generator by $\widehat{\Delta}_p$. We also denote the Dirichlet form on $L^2(\widehat{X};\widehat{\m})$ associated with $\widehat{\Delta}_2$ by 
$(\widehat{\mathscr{E}}, D(\widehat{\mathscr{E}}))$. That is, 
\begin{align*}
D(\widehat{\mathscr{E}}):&=\left\{u\in L^2(\widehat{X};\widehat{\m})\;\left|\; \lim_{t\to0}\frac{1}{t}(u-\widehat{P}_tu,u)_{L^2(\widehat{X};\widehat{\m})}<+\infty\right.\right\},\\
\widehat{\mathscr{E}}(u,v):&=\lim_{t\to0}\frac{1}{t}(u-\widehat{P}_tu,v)_{L^2(\widehat{X};\widehat{\m})}\qquad\text{ for }\quad u,v\in D(\widehat{\mathscr{E}}).
\end{align*} 
Since $(\widehat{\mathscr{E}}, D(\widehat{\mathscr{E}}))$ is associated to the 
$\widehat{\m}$-symmetric Borel right process $\widehat{\bf X}$, it is quasi-regular by 
Fitzsimmons~\cite{Fitzsimmons}. 

Throughout this subsection, we assume $\kappa^+\in S_D({\bf X})$ and $\kappa^-\in S_{E\!K}({\bf X})$.  
We define $\widehat{\kappa}^{\,\pm}:=\kappa^{\pm}\otimes m$ and $\widehat{\kappa}:=
\widehat{\,\kappa}^+-\widehat{\,\kappa}^-$. 
The associated positive continuous additive functional (PCAF in short) $\widehat{A}_t^{\;\widehat{\,\kappa}^{\pm}}$ in Revuz correspondence under $\wh{\bf X}$ is given by $\widehat{A}_t^{\;\widehat{\kappa}^{\pm}}=A_t^{\kappa^{\pm}}$, in particular, $\widehat{\kappa}^{\,+}$ (resp.~$\widehat{\kappa}^{\,-}$) 
is a smooth measure of Dynkin (resp.~extended Kato) class  with respect to $\widehat{\bf X}$, i.e. 
$\widehat{\kappa}^{\,+}\in S_D(\widehat{\bf X})$ (resp.~$\widehat{\kappa}^{\,-}\in S_{E\!K}(\widehat{\bf X})$). 
Then we can define the following quadratic form 
$(\widehat{\mathscr{E}}^{\;\widehat{\kappa}}, D(\widehat{\mathscr{E}}^{\;\widehat{\kappa}}))$ 
on $L^2(\widehat{X};\widehat{\m})$: 
\begin{align}
D(\widehat{\mathscr{E}}^{\;\widehat{\kappa}}):=D(\widehat{\mathscr{E}}),\qquad
\widehat{\mathscr{E}}^{\;\widehat{\kappa}}(u,v):=\widehat{\mathscr{E}}(u,v)+\langle\widehat{\kappa},\tilde{u}\tilde{v}\rangle\quad\text{ for }\quad u,v\in D(\widehat{\mathscr{E}}^{\;\widehat{\kappa}}).\label{eq:quadratic}
\end{align}
Here $\tilde{u}$ denotes the $\widehat{\mathscr{E}}$-quasi-continuous $\widehat{\m}$-version of $u$ with respect to 
$(\widehat{\mathscr{E}}, D(\widehat{\mathscr{E}}))$. 
The strongly continuous semigroup $(\widehat{P}_t^{\;\widehat{\kappa}})_{t\geq0}$ on $L^2(\widehat{X};\widehat{\m})$ associated with the quadratic form 
$(\widehat{\mathscr{E}}^{\;\widehat{\kappa}}, D(\widehat{\mathscr{E}}^{\;\widehat{\kappa}}))$ 
is given by 
\begin{align}
\widehat{P}_t^{\widehat{\;\kappa}}u(\hat{x})=
\E_{\hat{x}}
[e^{-A_t^{\kappa}}u(\widehat{X}_t)]
\quad \text{ for }\quad u\in L^2(\widehat{X};\widehat{\m})\cap \mathscr{B}(\widehat{X}).\label{eq:FeynmanKac}
\end{align}

We denote by $\widehat{\mathscr{C}}:={\rm Test}
(X)\otimes C_c^{\infty}(\R)$ the totality of all linear combinations of $f\otimes\varphi$, $f\in {\rm Test}(X)$, $\varphi\in C_c^{\infty}(\R)$, where 
$(f\otimes \varphi)(y):=f(x)\varphi(a)$. Meanwhile, the space $L^2(X;\m)\otimes L^2(\R)$ and 
$D(\mathscr{E})\otimes H^{1,2}(\R)$ are usual tensor products of Hilbert spaces, where $H^{1,2}(\R)$ is the Sobolev space which consists of all functions $\varphi\in L^2(\R)$ such that the weak derivative $\varphi'$ exists and belongs to $L^2(\R)$. Then we have 

\begin{lem}\label{lem:Core}
$\widehat{\mathscr{C}}$ is dense in $D(\widehat{\mathscr{E}}
)$. Moreover, 
for $u,v\in D(\mathscr{E}
)\otimes H^{1,2}(\R)$,  we have
\begin{align}
\widehat{\mathscr{E}}(u,v)&=\int_{\R}\mathscr{E}(u(\cdot,a),v(\cdot,a))m(\d a)+
\int_X\m(\d x)\int_{\R}\frac{\partial u}{\partial a}(x,a)\frac{\partial v}{\partial a}(x,a)m(\d a).\label{eq:Identity}
\end{align}
\end{lem}
\begin{proof}[\bf Proof]  
We denote by $(T_t)_{t\geq0}$ the transition semigroup associated with $(B_t,\P_a^{\rightarrow})$. 
We can regard it as the semigroup on $L^2(\R)$. First, we denote that the following identity holds:
\begin{align}
\widehat{P}_t(f\otimes\varphi)=(P_tf)\otimes(T_t\varphi), \quad f\in L^2(X;\m),\quad \varphi\in L^2(\R).\label{eq:productSemigroup}
\end{align} 
By \eqref{eq:productSemigroup}, we can see $\widehat{\mathscr{C}}\subset D(\mathscr{E})\otimes H^{1,2}(\R)\subset D(\widehat{\mathscr{E}})$ and the identity \eqref{eq:Identity}. We also have
\begin{align}
\widehat{\mathscr{E}}_1(f\otimes \varphi,f\otimes\varphi)=\mathscr{E}(f,f)\|\varphi\|_{L^2(\R)}^2
+\|f\|_{L^2(X;\m)}^2\left(\|\varphi'\|_{L^2(\R)}^2+\|\varphi\|_{L^2(\R)}^2 \right)\label{eq:CartesProduct}
\end{align} 
holds for $f\in D(\mathscr{E})$, $\varphi\in H^{1,2}(\R)$. By \eqref{eq:CartesProduct}, 
we can see that $\widehat{\mathscr{C}}$ is dense in $D(\mathscr{E})\otimes H^{1,2}(\R)$ with respect 
to $\widehat{\mathscr{E}}_1$, because ${\rm Test}(X)$ and $C_c^{\infty}(\R)$ are dense in $(\mathscr{E},D(\mathscr{E}))$ and in $H^{1,2}(\R)$, respectively. 

Hence it is sufficient to show $D(\mathscr{E})\otimes H^{1,2}(\R)$ is dense in $(\widehat{\mathscr{E}}, D(\widehat{\mathscr{E}}))$. Since $L^2(X;\m)\otimes L^2(\R)$ is dense in $L^2(\widehat{X};\widehat{\m})$, 
$\bigcup_{t>0}\widehat{P}_t\left(L^2(X;\m)\otimes L^2(\R) \right)$ is dense in 
$(\widehat{\mathscr{E}}, D(\widehat{\mathscr{E}}))$. On the other hand, 
\eqref{eq:productSemigroup} also leads us to 
\begin{align*}
\bigcup_{t>0}\widehat{P}_t\left(L^2(X;\m)\otimes L^2(\R) \right)&=\bigcup_{t>0}\left(P_t(L^2(X;\m)) \right)
\otimes \left(T_t(L^2(\R))\right)\\
&\subset D(\mathscr{E})\otimes H^{1,2}(\R)\subset D(\widehat{\mathscr{E}}).
\end{align*}
This completes the proof.
\end{proof} 
Denote by $\mathscr{P}(X)$, the family of all Borel probability measures on $X$, and  
by $\mathscr{B}^*(X)$, the family of all universally measurable sets, that is, $\mathscr{B}^*(X):=\bigcap_{\nu\in\mathscr{P}(X)}\overline{\mathscr{B}(X)}^{\nu}$, where $\overline{\mathscr{B}(X)}^{\nu}$ is the $\nu$-completion of $\mathscr{B}(X)$ for $\nu\in\mathscr{P}(X)$. $\mathscr{B}^*(X)$ also denotes the family of 
universally measurable real valued functions. Moreover, $\mathscr{B}_b^*(X)$ (resp.~$\mathscr{B}_+^*(X)$) denotes the family of bounded (resp.~non-negative) universally measurable functions.  
For $f\in \mathscr{B}_b^*(X)$, or $f\in \mathscr{B}_+^*(X)$, we set
\begin{align}
q_t^{(\alpha),\kappa}f(x):=\E_x\left[\int_0^{\infty}e^{-\alpha s-A_s^{\kappa}}f(X_s)\lambda_t(\d s) \right]
\end{align}
and write $q_t^{(\alpha)}f(x)$ instead of $q_t^{(\alpha),0}f(x)$. Then $q_t^{(\alpha),\kappa}f\in \mathscr{B}_b^*(X)$ (resp.~$q_t^{(\alpha),\kappa}f\in \mathscr{B}_+^*(X)$) under $f\in \mathscr{B}_b^*(X)$ (resp.~$f\in \mathscr{B}_+^*(X)$).  
It is easy to see $q_t^{(\alpha),\kappa}1(x)\leq Ce^{-\sqrt{\alpha-C_{\kappa}}t}$ ($\alpha\geq C_{\kappa}$) and 
$q_t^{(\alpha)}1(x)\leq e^{-\sqrt{\alpha}t}$ ($\alpha\geq0$). 
Take $f\in L^2(X;\m)\cap \mathscr{B}_b^*(X)$ or $f\in L^2(X;\m)\cap \mathscr{B}_+^*(X)$.  
According to the equation 
\begin{align}
(q_t^{(\alpha),\kappa}f,g)_{\m}&=\int_0^{\infty}e^{-\alpha s}(p_s^{\kappa}f,g)_{\m}\lambda_t(\d s)\notag\\
&=\int_0^{\infty}e^{-\alpha s}(P_s^{\kappa}f,g)_{\m}\lambda_t(\d s)\label{eq:Version}\\
&=(Q_t^{(\alpha),\kappa}f,g)_{\m}\quad\text{ for any }\quad g\in L^2(X;\m),\notag
\end{align}
$q_t^{(\alpha),\kappa}f$ is an $\m$-version of $Q_t^{(\alpha),\kappa}f$.
Now we fix a function $f\in D(\Delta^{\kappa})
\cap \mathscr{B}^*(X)$. 
We  set $u(x,a):=q_a^{(\alpha),\kappa}f(x)$ $(\alpha\geq C_{\kappa})$. Then it holds that 
\begin{align}
&a\mapsto u(\cdot,a)\quad\text{ is an $L^2(X;\m)$-valued smooth function, and }\notag\\
&\left(\frac{\partial^2}{\partial a^2}+\Delta^{\kappa}
-\alpha \right)u(\cdot,a)=0\quad\text{ in }\quad L^2(X;\m).\label{eq:recursive}
\end{align}
Furthermore for $a\in\R$, we consider $v(x,a):=u(x,|a|)=q_{|a|}^{(\alpha),\kappa
}f(x)$ for $\alpha>C_{\kappa}$. 
Then by \eqref{eq:Contra}, we have 
\begin{align}
\|v\|_{L^2(\widehat{X};\widehat{\m})}\leq \left(\int_{\R}Ce^{-2\sqrt{\alpha-C_{\kappa}
}|a|}\|f\|_{L^2(X;\m)}^2 \d a\right)^{1/2}=C^{\frac12}(
\alpha-C_{\kappa}
)
^{-1/4}\|f\|_{L^2(X;\m)}.\label{eq:L^2Est}
\end{align}
The main purpose of this subsection is the semimartingale decomposition of $\tilde{v}(\widehat{X}_{t\land\tau})$, $t\geq0$, where $\tau:=\inf\{t>0\mid B_t=0\}$. As 
the first step, we give the following: 
\begin{lem}\label{lem:Domain}
For $f\in D(\Delta^{\kappa})
\cap \mathscr{B}^*(X)$ and $\alpha>C_{\kappa}$, 
$v(x,a):=q_{|a|}^{(\alpha),\kappa}f(x)$ satisfies 
$v\in D(\widehat{\mathscr{E}})$. 
\end{lem}
\begin{proof}[\bf Proof]  
At the beginning, we note that $L^2(\widehat{X};\widehat{\m})\cong L^2(\R,L^2(X;\m);m)$. 
We may assume $f$ is non-negative on $X$, because there exists $g\in L^2(X;\m)\cap \mathscr{B}(X)$ such that $f=R_{\alpha}^{\kappa}g$ for some/any $\alpha>C_{\kappa}$ and $f$ is difference of 
the elements in $D(\Delta^{\kappa})
\cap \mathscr{B}^*(X)_+$ by setting $g^{\pm}:=\max\{\pm g,0\}$. Here $R_{\alpha}^{\kappa}g(x):=\E_x\left[\int_0^{\infty}e^{-\alpha t-A_t^{\kappa}}g(X_t)\d t\right]$ is an $\m$-version of 
the $\alpha$-order resolvent 
$G_{\alpha}^{\kappa}g$ $(\alpha>C_{\kappa})$ associated to $(\mathscr{E}^{\kappa},D(\mathscr{E}^{\kappa}))$.
According to Fubini's theorem, we have 
\begin{align}
\widehat{P}_t^{\,\widehat{\kappa}}
v(\hat{x})=\E_{\hat{x}}\left[e^{-A_t^{\kappa}}
u(X_t,|B_t|) \right]=\E_x\left[e^{-A_t^{\kappa}}
\E_a^{\rightarrow}[u(\cdot,|B_t|)] \right].\label{eq:Semigroup}
\end{align}
We recall Tanaka's formula
\begin{align*}
|B_t|=|B_0|+\int_0^t{\rm sgn}(B_s)\d B_s+L_t^0,\quad t\geq0,\quad \P_a^{\rightarrow}\text{-a.s.,}
\end{align*}
where $(L_t^0)_{t\geq0}$ is the local time of one-dimensional Brownian motion $(B_t)_{t\geq0}$ at the origin. Then by using It\^o's formula, we have 
\begin{align}
u(\cdot,|B_t|)&=u(\cdot,|B_0|)+\int_0^t
\frac{\partial u}{\partial a}(\cdot, |B_s|){\rm sgn}(B_s)\d B_s
\label{eq:Ito}\\
&\hspace{1cm}+\int_0^t\frac{\partial u}{\partial a}(\cdot, |B_s|)%{\rm sgn}(B_s)
\d L_s^0+
\int_0^t\frac{\partial^2 u}{\partial a^2}(\cdot,|B_s|)\d s\notag\\
&=u(\cdot,|B_0|)-\int_0^t\sqrt{\alpha-\Delta^{\kappa}
} u(\cdot,|B_s|){\rm sgn}(B_s)\d B_s\notag\\
&\hspace{1cm}-\int_0^t\sqrt{\alpha-\Delta^{\kappa}
} u(\cdot,|B_s|)
\d L_s^0
+\int_0^t(\alpha-\Delta^{\kappa}
)u(\cdot,|B_s|)\d s.
\notag
\end{align}
Hence \eqref{eq:Ito} leads us to 
\begin{align}
\E_a^{\rightarrow}[u(\cdot,|B_t|)]&=u(\cdot,|a|)-\E_a^{\rightarrow}\left[\int_0^t\sqrt{\alpha-\Delta^{\kappa}
}u(\cdot,|B_s|)\d L_s^0 \right]\label{eq:OneIto}\\
&\hspace{1cm}+\E_a^{\rightarrow}\left[\int_0^t(\alpha-\Delta^{\kappa}
)u(\cdot,|B_s|)\d s \right].\notag
\end{align}
On the other hand, $u(\cdot,|a|)=q_{|a|}^{(\alpha),\kappa
}f(\cdot)\in D(\Delta^{\kappa}
)$ in view of 
\eqref{eq:recursive}. Hence
\begin{align*}
M_t^{[u(\cdot,|a|)]}:=e^{-A_t^{\kappa}}
(q_{|a|}^{(\alpha),\kappa
}f)(X_t)-(q_{|a|}^{(\alpha),\kappa
}f)(X_0)
-\int_0^te^{-A_s^{\kappa}}
\Delta^{\kappa}
(q_{|a|}^{(\alpha),\kappa
}f)(X_s)\d s, \quad t\geq0,
\end{align*} 
is an $L^2(\P_x)$-martingale. Then we have
\begin{align}
\E_x\left[e^{-A_t^{\kappa}}
u(X_t,|a|) \right]&=(q_{|a|}^{(\alpha),\kappa
})f(x)+\int_0^tp_s^{\kappa}
(\Delta^{\kappa}
q_{|a|}^{(\alpha),\kappa
}f)(x)\d s, \quad\m\text{-a.e.~}x\in X.
\label{eq:VerticlaEq}
\end{align}
By summarizing \eqref{eq:Semigroup}, \eqref{eq:OneIto} and \eqref{eq:VerticlaEq}, we can proceed as 
\begin{align}
\frac{1}{t}&(v-\widehat{P}_t^{\widehat{\;\kappa}}
v,v)_{L^2(\widehat{X};\widehat{\m})}\notag\\
&=-\frac{1}{t}\int_{\R}\d a\int_X\left\{\int_0^t P_s^{\kappa}
(\Delta^{\kappa}
Q_{|a|}^{(\alpha),\kappa
}f)(x)\d s \right\}\cdot Q_{|a|}^{(\alpha),\kappa
}f(x)\m(\d x)\notag\\
&\hspace{1cm}+
\frac{1}{t}\int_{\R}\d a\int_X\E_x\left[e^{-A_t^{\kappa}}
\E_a^{\rightarrow}\left[
\int_0^t\sqrt{\alpha-\Delta^{\kappa}
}u(\cdot,|B_s|)\d L_s^0 \right](X_t) \right]
\cdot Q_{|a|}^{(\alpha),\kappa
}f(x)\m(\d x)\notag\\
&\hspace{1cm}-
\frac{1}{t}\int_{\R}\d a\int_X\E_x\left[e^{-A_t^{\kappa}}
\E_a^{\rightarrow}\left[
\int_0^t(\alpha-\Delta^{\kappa}
)u(\cdot,|B_s|)\d L_s^0 \right](X_t) \right]
\cdot Q_{|a|}^{(\alpha),\kappa
}f(x)\m(\d x)
\notag\\
&=-\frac{1}{t}\int_{\R}\d a\int_0^t(P_s^{\kappa}
\Delta^{\kappa}
Q_{|a|}^{(\alpha),\kappa
}f,Q_{|a|}^{(\alpha),\kappa
}f)_{L^2(X;\m)}\d s\label{eq;Dirichlet}\\
&\hspace{1cm}+
\frac{1}{t}\int_{\R}\d a\int_X\E_a^{\rightarrow}\left[\int_0^t\sqrt{\alpha-\Delta^{\kappa}
}u(x,|B_s|)\d L_s^0 \right]P_t^{\kappa}
(Q_{|a|}^{(\alpha),\kappa
}f)(x)\m(\d x)\notag\\
&\hspace{1cm}-
\frac{1}{t}\int_{\R}\d a\int_X\E_a^{\rightarrow}\left[\int_0^t(\alpha-\Delta^{\kappa}
)u(x,|B_s|)\d s \right]P_t^{\kappa}(Q_{|a|}^{(\alpha),\kappa
}f)(x)\m(\d x)\notag\\
&=:-I_1(t)+I_2(t)-I_3(t),\notag
\end{align}
where we used the symmetry of $(P_t^{\kappa}
)_{t\geq0}$ on $L^2(X;\m)$. For the terms $I_1(t)$ and $I_2(t)$,  we see the following estimates by using the contractivity \eqref{eq:KatoContraction} 
of $(P_t^{\kappa}
)_{t\geq0}$ on $L^2(X;\m)$ and  \eqref{eq:Contra}:
\begin{align}
|I_1(t)|&\leq\frac{1}{t}\int_{\R}\d a\int_0^tCe^{C_{\kappa}s}
\|\Delta^{\kappa}Q_{|a|}^{(\alpha),\kappa
}f\|_{L^2(X;\m)}\cdot\|Q_{|a|}^{(\alpha),\kappa
}f \|_{L^2(X;\m)}\d s\notag\\
&\leq 
\left(\frac{C
}{t}\int_0^te^{{C_{\kappa}s}}
\d s \right)
\int_{\R}e^{-2\sqrt{\alpha-C_{\kappa}
}|a|}
\|\Delta^{\kappa}
f\|_{L^2(X;\m)}\cdot\|f\|_{L^2(X;\m)}\d a\label{eq:I_1}\\
&\leq 
\left(\frac{C
}{t}\int_0^te^{{C_{\kappa}s}}
\d s \right)
\frac{1}{\sqrt{\alpha-C_{\kappa}
}}
\|\Delta^{\kappa}
f\|_{L^2(X;\m)}\cdot\|f\|_{L^2(X;\m)},\notag
\end{align}
\begin{align*}
|I_2(t)|&=\left|\frac{1}{t}\int_{\R}\d a\int_X\left( \sqrt{\alpha-\Delta^{\kappa}
}u(x,0)\E_a^{\rightarrow}\left[L_t^0 \right]\right)\cdot P_t^{\kappa}
(Q_{|a|}^{(\alpha),\kappa
}f)(x)\m(\d x) \right|\\
&=\frac{1}{t}\left|\int_{\R}(\sqrt{\alpha-\Delta^{\kappa}
}f, P_t^{\kappa}
Q_{|a|}^{(\alpha),\kappa
}f)_{L^2(X;\m)}\E_a^{\rightarrow}[L_t^0]\d a \right|\\
&\leq \frac{2C
}{t}e^{C_{\kappa}t}
\|\sqrt{\alpha-\Delta^{\kappa}
}f\|_{L^2(X;\m)}\cdot\|f\|_{L^2(X;\m)}
\int_0^{\infty}e^{-\sqrt{\alpha-C_{\kappa}
}a}\E_a^{\rightarrow}[L_t^0]\d a.
\end{align*}
Here we recall 
\begin{align*}
\P_a^{\rightarrow}(L_t^r\in\d y)=\frac{1}{\sqrt{\pi t}}\exp\left( -\frac{(y+|r-a|)^2}{4t}\right)\d y,\qquad y>0.
\end{align*}
(see \cite[p.~155]{BorodinSalminen}). Then we can continue as 
\begin{align}
|I_2(t)|&\leq \frac{2C
}{t}e^{C_{\kappa}t}
\|\sqrt{\alpha-\Delta^{\kappa}
}f\|_{L^2(X;\m)}\cdot\|f\|_{L^2(X;\m)}\notag\\
&\hspace{1cm}\times \int_0^{\infty}e^{-\sqrt{\alpha-C_{\kappa}
}a}\left\{\int_0^{\infty}y\frac{1}{\sqrt{\pi t}}\exp\left(-\frac{(a+y)^2}{4t} \right)\d y \right\}\d a\notag\\
&\leq 8C^{\frac32}e^{C_{\kappa}t}
\|\sqrt{\alpha-\Delta^{\kappa}
}f\|_{L^2(X;\m)}
\cdot
\|f\|_{L^2(X;\m)}\label{eq:I_2}\\
&\hspace{1cm}\times \int_0^{\infty}\frac{1}{\sqrt{2\pi}}e^{-\frac{a^2}{2}}\d a\int_0^{\infty}y e^{-\frac{y^2}{2}}\d y\notag\\
&=4C^{\frac32}e^{C_{\kappa}t}
\|\sqrt{\alpha-\Delta^{\kappa}
}f\|_{L^2(X;\m)}\cdot\|f\|_{L^2(X;\m)}.\notag
\end{align}
For the term $I_3(t)$, we have
\begin{align}
|I_3(t)|&\leq \frac{Ce^{C_{\kappa}t}
}{t}
\int_{\R}\left\|
\E_a^{\rightarrow}
\left[\int_0^t(\alpha-\Delta^{\kappa}
)u(\cdot,|B_s|)\d s \right]
 \right\|_{L^2(X;\m)} \|Q_{|a|}^{(\alpha),\kappa
 } f \|_{L^2(X;\m)}\d a\notag\\
 &\leq
 \frac{Ce^{C_{\kappa}t}
 }{t}\int_{\R}\E_a^{\rightarrow}\left[\int_0^t\|(\alpha-\Delta^{\kappa}
 )Q_{|B_s|}^{(\alpha),\kappa
 }f(\cdot) \|_{L^2(X;\m)}\d s \right]
\left(e^{-\sqrt{\alpha- C_{\kappa}
}|a|}\|f\|_{L^2(X;\m)} \right)\d a\label{eq:I_3}\\
&\leq \frac{Ce^{C_{\kappa}t}
}{t}\int_{\R}\E_a^{\rightarrow}\left[\int_0^t(\alpha\|f\|_{L^2(X;\m)}+\|\Delta^{\kappa}f
\|_{L^2(X;\m)})\d s \right]
\left(e^{-\sqrt{\alpha- C_{\kappa}
}|a|}\|f\|_{L^2(X;\m)} \right)\d a\notag\\
&\leq \frac{2Ce^{C_{\kappa}t}
}{\sqrt{\alpha-C_{\kappa}
}}
(\alpha \|f\|_{L^2(X;\m)}+\|\Delta^{\kappa}
f\|_{L^2(X;\m)})\|f\|_{L^2(X;\m)}\notag
\end{align}
Finally, we substitute estimates \eqref{eq:I_1}, \eqref{eq:I_2} and \eqref{eq:I_3} into \eqref{eq;Dirichlet}. Then we can easily see 
\begin{align*}
\lim_{t\to0}\frac{1}{t}(v-\widehat{P}_t^{\,\widehat{\kappa}}
v,v)_{L^2(\widehat{X};\widehat{\m})}
=\sup_{t>0}\frac{1}{t}(v-\widehat{P}_t^{\,\widehat{\kappa}}
v,v)_{L^2(\widehat{X};\widehat{\m})}<\infty.
\end{align*}
This and \eqref{eq:L^2Est} with \eqref{eq:quadratic} complete the proof. 
\end{proof}

By Lemma~\ref{lem:Domain}, we can apply Fukushima's decomposition to 
$v(x,a):=q_{|a|}^{(\alpha),\kappa}f(x)$ for $f\in D(\Delta^{\kappa})\cap \mathscr{B}^*(X)$ and $\alpha>C_{\kappa}$. 
That is, there exists 
a martingale additive functional of finite energy $\wh{M}^{[v]}$ and a continuous additive functional of zero energy $\wh{N}^{[v]}$ such that 
\begin{align}
\tilde{v}(\widehat{X}_t)-\tilde{v}(\widehat{X}_0)=\wh{M}_t^{[v]}+\wh{N}_t^{[v]}\quad t\geq0\quad \P_{\hat{x}}\text{-a.s.~for q.e.~}\hat{x},\label{eq:FukushimaDecomp} 
\end{align}
where $\tilde{v}$ is an $\widehat{\mathscr{E}}$-quasi-continuous $\widehat{\m}$-version of 
$v\in D(\widehat{\mathscr{E}}^{\,\kappa})=
 D(\widehat{\mathscr{E}})$.
See \cite[Theorem~5.2.2]{FOT} for Fukushima's decomposition theorem. Thanks to \cite[Theorem~5.2.3]{FOT},  we know that 
\begin{align}
\langle \wh{M}^{[v]}\rangle_t=\int_0^t \left\{\Gamma(v,v)(\widehat{X}_s)+\left( \frac{\partial v}{\partial a}(\widehat{X}_s)\right)^2 \right\}\d s,\quad t\geq0.\label{eq:quadraticVariation}
\end{align}
See also \cite[Theorem~5.1.3 and Example~5.2.1]{FOT} for details.  

From now on, we give the explicit expression of $\wh{N}^{[v]}$. Let us define the signed measure $\nu$ on $\widehat{X}$ by 
\begin{align*}
\nu(\d x\d a):=2\sqrt{\alpha-\Delta^{\kappa}
} v(x,a)\m(\d x)\delta_0(\d a)
\end{align*}
for $\alpha>C_{\kappa}$, 
where $\delta_0$ is Dirac measure on $\R$ with unit mass at origin. The total variation of 
$\nu$ is given by 
\begin{align*}
|\nu|(\d x\d a):=2|\sqrt{\alpha-\Delta^{\kappa}
}v(x,a)|\m(\d x)\delta_0(\d a).
\end{align*}
Note here that $\nu$ depends on $f\in D(\Delta^{\kappa})\cap \mathscr{B}^*(X)$ and $\alpha>C_{\kappa}$. 
Then we have
\begin{lem}\label{lem:finiteEnergy}
Suppose $f\in D(\Delta^{\kappa})\cap \mathscr{B}^*(X)$ and $\alpha>C_{\kappa}$. Then, there 
exists a constant $C>0$ such that 
\begin{align*}
\int_{\widehat{X}}|g\otimes \varphi(x,a)|\cdot|\nu|(\d x\d a)\leq C\sqrt{\widehat{\mathscr{E}}_1(g\otimes\varphi,g\otimes\varphi)}
\end{align*}
for $g\in{\rm Test}
(X)$, $\varphi\in C_c^{\infty}(\R)$. That is, $\nu$ is of finite energy integral. For instance, see \cite[Sections~2.2 and 5.4]{FOT} for the definitions of measures 
of finite energy integrals. 
\end{lem}
\begin{proof}[\bf Proof]  
We take a positive constant $a_0$ such that ${\rm supp}[\varphi]\subset [-a_0,a_0]$. 
We first consider the case of $\varphi(0)\leq0$. Let $\eps>0$. Then for $\m$-a.e.~$x\in X$, we have 
\begin{align*}
\int_{\R}&|\varphi(a)|\sqrt{(\sqrt{\alpha-\Delta^{\kappa}}v(x,a))^2+\eps}\delta_0(\d a)\\
&=-\varphi(0)\sqrt{(\sqrt{\alpha-\Delta^{\kappa}}v(x,0))^2+\eps}\\
&=\varphi(a_0)\sqrt{(\sqrt{\alpha-\Delta^{\kappa}}v(x,a_0))^2+\eps}-\varphi(0)\sqrt{(\sqrt{\alpha-\Delta^{\kappa}}v(x,0))^2+\eps}\\
&=\int_0^{a_0}\frac{\partial}{\partial a}\left\{\varphi(a)\sqrt{(\sqrt{\alpha-\Delta^{\kappa}}v(x,a))^2+\eps} \right\}\d a\\
&=\int_0^{a_0}\varphi'(a)\sqrt{(\sqrt{\alpha-\Delta^{\kappa}}v(x,a))^2+\eps}\d a\\
&\hspace{1cm}-\int_0^{a_0}\varphi(a)\frac{\sqrt{\alpha-\Delta^{\kappa}}v(x,a)\cdot(\alpha-\Delta^{\kappa})v(x,a)}{\sqrt{(\sqrt{\alpha-\Delta^{\kappa}}v(x,a_0))^2+\eps}}\d a\\
&\leq \int_{\R}|\varphi'(a)|\sqrt{(\sqrt{\alpha-\Delta^{\kappa}}v(x,a_0))^2+\eps} \d a+
\int_{\R}|\varphi(a)|\cdot|(\alpha-\Delta^{\kappa})v(x,a)|\d a.
\end{align*}
Therefore
\begin{align*}
\int_{\widehat{X}}&|(g\otimes \varphi)(x,a)|\cdot|\nu|(\d x\d a)\\
&=2\lim_{\eps\to0}\int_X|g(x)|\left(\int_{\R}|\varphi(a)|
\sqrt{(\sqrt{\alpha-\Delta^{\kappa}}v(x,a))^2+\eps}\delta_0(\d a)
 \right)\m(\d x)\\
 &\leq
 2\varlimsup_{\eps\to0}
 \int_X|g(x)|\left(\int_{\R}|\varphi'(a)|
\sqrt{(\sqrt{\alpha-\Delta^{\kappa}}v(x,a))^2+\eps}\d a
 \right)\m(\d x)\\
 &\hspace{1cm}+2\int_X|g(x)|\left(\int_{\R}|\varphi(a)|\cdot|(\alpha-\Delta^{\kappa})v(x,a)|\d a \right)\m(\d x)\\
 &\leq 2\left(
 \|\sqrt{\alpha-\Delta^{\kappa}}v\|_{L^2(\widehat{X};\widehat{\m})}\|\varphi'\|_{L^2(\R)}+
  \|(\alpha-\Delta^{\kappa})v\|_{L^2(\widehat{X};\widehat{\m})}\|\varphi\|_{L^2(\R)}
  \right)\|g\|_{L^2(X;\m)}\\
  &\leq 2\sqrt{C}(\alpha-C_{\kappa})^{-1/4}\left(\|\sqrt{\alpha-\Delta^{\kappa}}f\|_{L^2(\widehat{X};\widehat{\m})}+
  \|(\alpha-\Delta^{\kappa})f\|_{L^2(\widehat{X};\widehat{\m})}\right)\sqrt{\widehat{\mathscr{E}}_1(g\otimes\varphi,g\otimes\varphi)}\\
  &=:C\sqrt{\widehat{\mathscr{E}}_1(g\otimes\varphi,g\otimes\varphi)},
\end{align*}
where we used \eqref{eq:L^2Est} and 
\begin{align*}
\widehat{\mathscr{E}}(g\otimes \varphi,g\otimes\varphi)=\mathscr{E}(g,g)\|\varphi\|_{L^2(\R)}^2+\|g\|_{L^2(X;\m)}^2\|\varphi'\|_{L^2(\R)}^2
\end{align*}
for the last line. This is the desired result. 

In the case of $\varphi(0)\geq0$, we easily see 
\begin{align}
\int_{\R}|\varphi(a)|\sqrt{(\sqrt{\alpha-\Delta^{\kappa}}v(x,a))^2+\eps}\delta_0(\d a)=
\int_{-a_0}^0\frac{\partial}{\partial a}\left\{\varphi(a)\sqrt{(\sqrt{\alpha-\Delta^{\kappa}}v(x,a))^2+\eps} \right\}\d a.\label{eq:negativeCase}
\end{align}
By using \eqref{eq:negativeCase}, we can follow the same argument as the case where $\varphi(0)\leq0$. Therefore the proof is completed. 
\end{proof} 

Thanks to Lemma~\ref{lem:finiteEnergy}, $\nu$ is of finite $1$-order energy integral. Then for 
each $\beta>0$, there exists a unique $U_{\beta}\nu\in D(\mathscr{E})$ such that 
the following relation holds: 
\begin{align}
\widehat{\mathscr{E}}_{\beta}(U_{\beta}\nu,g\otimes\varphi)=\int_{\widehat{X}}(g\otimes\varphi)(x,a)\nu(\d x\d a),\quad g\in {\rm Test}(X),\varphi\in C_c^{\infty}(\R).\label{eq:Potential}
\end{align}
\begin{lem}\label{lem:Potential}
We have the following under $f\in D(\Delta^{\kappa})\cap \mathscr{B}^*(X)$, $\alpha>C_{\kappa}$ and $\beta>0$: 
\begin{enumerate}
\item[{\rm (1)}] $U_{\alpha}\nu=v$. 
\item[{\rm (2)}] $U_{\beta}\nu=v-(\beta-\alpha)\widehat{R}_{\beta}v$ holds, where $(\widehat{R}_{\beta})_{\beta>0}$ is the resolvent of $(\widehat{P}_t)_{t\geq0}$. 
\end{enumerate}
\end{lem}
\begin{proof}[\bf Proof]  
(1) We need to show \eqref{eq:Potential}. By using the integration by parts formula, 
for $\m$-a.e.~$x\in X$, we have 
\begin{align}
\int_{\R}&\frac{\partial v}{\partial a}(x,a)\varphi'(a)\d a\label{eq:Poten1}\\
&=-\int_0^{\infty}\sqrt{\alpha-\Delta^{\kappa}}u(x,a)\varphi'(a)\d a+\int_0^{\infty}
\sqrt{\alpha-\Delta^{\kappa}}u(x,a)\varphi'(-a)\d a\notag\\
&=-\int_0^{\infty}\sqrt{\alpha-\Delta^{\kappa}}u(x,a)\frac{\d}{\d a}(\varphi(a)+\varphi(-a))\d a\notag\\
&=2\sqrt{\alpha-\Delta^{\kappa}}u(x,0)\varphi(0)+\int_0^{\infty}\frac{\partial}{\partial a}
\sqrt{\alpha-\Delta^{\kappa}}u(x,a)(\varphi(a)+\varphi(-a))\d a\notag\\
&=2\sqrt{\alpha-\Delta^{\kappa}}u(x,0)\varphi(0)-
\int_0^{\infty}(\alpha-\Delta^{\kappa})u(x,a)(\varphi(a)+\varphi(-a))\d a\notag\\
&=2\sqrt{\alpha-\Delta^{\kappa}}v(x,0)\varphi(0)-
\int_{\R}(\alpha-\Delta^{\kappa})v(x,a)\varphi(a)\d a.\notag
\end{align}
Then \eqref{eq:Poten1} leads us to our desired equality as follows: 
\begin{align*}
\widehat{\mathscr{E}}_{\alpha}(v,g\otimes\varphi)&=\int_{\R}\d a\varphi(a)
\int_X\sqrt{\alpha-\Delta^{\kappa}}v(x,a)\sqrt{\alpha-\Delta^{\kappa}}g(x)\m(\d x)\\
&\hspace{1cm}+
\int_X\m(\d x)g(x)\left(2\sqrt{\alpha-\Delta^{\kappa}}v(x,0)\varphi(0)-\int_{\R}(\alpha-\Delta^{\kappa})v(x,a)\varphi(a)\d a \right)\\
&=2\int_X\sqrt{\alpha-\Delta^{\kappa}}v(x,0)g(x)\varphi(0)\m(\d x)\\
&=\int_{\widehat{X}}(g\otimes\varphi)(x,a)\nu(\d x\d a).
\end{align*}
(2) We recall $\widehat{\mathscr{E}}_{\beta}(\widehat{R}_{\beta}v,g\otimes\varphi)=(v,g\otimes\varphi)_{L^2(\widehat{X};\widehat{\m})}$. Then we have 
\begin{align*}
\widehat{\mathscr{E}}_{\beta}(v-(\beta-\alpha)\widehat{R}_{\beta}v,g\otimes\varphi)&=
\widehat{\mathscr{E}}_{\beta}(v,g\otimes\varphi)-(\beta-\alpha)\cdot(v,g\otimes\varphi)_{L^2(\widehat{X};\widehat{\m})}\\
&=\widehat{\mathscr{E}}_{\alpha}(v,g\otimes\varphi)\\
&=\int_{\widehat{X}}(g\otimes\varphi)(x,a)\nu(\d x\d a),
\end{align*}
where we used (1) for the last line. Hence the proof of (2) is now completed. 
\end{proof} 
Thanks to \cite[Lemma~5.4.1]{FOT} and Lemma~\ref{lem:Potential}, we have 
\begin{align*}
\wh{N}_t^{[v]}=\alpha\int_0^t\tilde{v}(\widehat{X}_s)\d s-A_t^{\nu},\quad t\geq0,
\end{align*}
where $\tilde{v}$ is an $\widehat{\mathscr{E}}$-quasi-continuous $\widehat{\m}$-version of $v$ and $A^{\nu}$ is the CAF corresponding to $\nu$. Since $\nu$ does not charge out of $X\times\{0\}$, due to \cite[Theorem~5.1.5]{FOT}, $A_{t\land\tau}^{\nu}=0$ holds. Thus we get
\begin{align}
\wh{N}^{[v]}_{t\land\tau}=\alpha\int_0^{t\land\tau}\tilde{v}(\widehat{X}_s)\d s.\label{eq:CAFzero}
\end{align}
By summarizing \eqref{eq:FukushimaDecomp}, \eqref{eq:quadraticVariation} and \eqref{eq:CAFzero}, 
we have the following semi-martingale decomposition which plays a crucial role later. 

\begin{prop}\label{prop:semimartingale}
Suppose $\alpha>C_{\kappa}$, $f\in D(\Delta^{\kappa})\cap \mathscr{B}^*(X)$ and set $v(x,a):=
q_{|a|}^{(\alpha),\kappa}f(x)$ for $(x,a)\in \widehat{X}$. Then
we have the semi-martingale decomposition
\begin{align}
\tilde{v}(\widehat{X}_{t\land\tau})-\tilde{v}(\widehat{X}_0)=\wh{M}_{t\land\tau}^{[v]}+\alpha\int_0^{t\land\tau}\tilde{v}(\widehat{X}_s)\d s,\quad t\geq0,\label{eq:semimartingale1}
\end{align}
under $\P_{\hat{x}}$ for q.e.~$\hat{x}=(x,a)$. Moreover it holds
\begin{align}
\langle \wh{M}^{[v]}\rangle_{t\land\tau}&=\int_0^{t\land\tau}\left\{\Gamma(v,v)(\widehat{X}_s)+
\left(\frac{\partial v}{\partial a}(\widehat{X}_s) \right)^2 \right\}\d s.\label{eq:semimartingale2}
\end{align}
In particular, by setting $M_t:=\wh{M}^{[v]}_{t\land\tau}$, 
\begin{align}
\E_{(x,a)}[\langle \wh{M}\rangle_{\infty}]&= \E_{(x,a)}
\left[\int_0^{\tau}\left\{\Gamma(v,v)+\left(\frac{\partial v}{\partial a}\right)^2\right\}(\widehat{X}_s)
\d s \right]<\infty\label{eq:semimartingale3}
\end{align}
for $\widehat{\m}$-a.e.~$(x,a)\in \widehat{G}:=X\times]0,+\infty[$, 
because the absorbing process $\widehat{\bf X}_{\widehat{G}}$ on 
$\widehat{G}$ is an $\widehat{\m}$-symmetric transient process and 
$\Gamma(v)+\left(\frac{\partial v}{\partial a} \right)^2\in L^1(\widehat{G};\widehat{\m})$.
\end{prop}

Since $v(x,a)=u(x,a)$ holds for $a\geq0$, this proposition also gives the semi-martingale decomposition 
of $u(\widehat{X}_{t\land\tau})$. 

Before closing this subsection, we need the following lemma to allow $\m\otimes\delta_a$ is an 
initial distribution. 
\begin{lem}\label{lem:Capacity}
We have the following:
\begin{enumerate}
\item[\rm(1)] Let $\{F_n\}$ be an $\wh{\mathscr{E}}$-nest of closed subsets of $\wh{X}$. Then 
$\{(F_n)_a\}$ is an $\mathscr{E}$-nest of closed subsets of $X$ for $m$-a.e.~$a\in\R$, where  
$(F_n)_a:=\{x\in X\mid (x,a)\in F_n\}$. 
\item[\rm(2)] Let $N$ be an $\wh{\mathscr{E}}$-exceptional set. Then, for $m$-a.e.~$a\in \R$, $N_a:=\{x\in X\mid (x,a)\in N\}$ is an $\mathscr{E}$-exceptional set. In particular, 
$\m\otimes\delta_a$ does not charge any $\widehat{\mathscr{E}}$-exceptional set for $m$-a.e.~$a\in\R$. 
\item[\rm(3)] Let $(x,a)\mapsto u(x,a)$ be an $\wh{\mathscr{E}}$-quasi continuous function. Then, for $m$-a.e.~$a\in \R$, 
$x\mapsto u(x,a)$ is an $\mathscr{E}$-quasi continuous function. 
\end{enumerate}
\end{lem}
\begin{proof}[\bf Proof]  
(2) and (3) are consequences of (1). We only prove (1). 
For $\m$-a.e.~strictly positive $\varphi\in L^2(X;\m)$ and $m$-a.e.~strictly positive $\phi\in L^2(\R)$, 
we set $\wh{h}:=G_1\varphi\otimes G_1^w\phi$, where $G_1$ (resp.~$G_1^w$) is the $1$-order resolvent operator on $L^2(X;\m)$ (resp.~$L^2(\R)$)  
associated to $(\mathscr{E},D(\mathscr{E}))$ (resp.~$(\mathbb{D},D(\mathbb{D}))$. 
Then $\wh{h}$ is a $2$-excessive function of $L^2(\wh{X};\wh{\m})$.    
Let $\{F_n\}$ be an $\wh{\mathscr{E}}$-nest of closed subsets of $\wh{X}$. 
Denote the $2$-order $\wh{h}$-weighted capacity by $\wh{\rm Cap}_{\wh{h},2}$ defined by 
\begin{align*}
\wh{\rm Cap}_{\wh{h},2}(O):=\inf\left\{\left.\wh{\mathscr{E}}_2(v,v)\,\right|\, 
v\in D(\wh{\mathscr{E}}),\,
v\geq \wh{h} \;\wh{\m}\text{-a.e.~on }O  \right\}
\end{align*}
for an open subset $O$ of $\wh{X}$. Since $\wh{h}\in D(\wh{\mathscr{E}})$ is $2$-excessive, 
we can deduce 
\begin{align}
\lim_{n\to\infty}\wh{\rm Cap}_{\wh{h},2}(\wh{X}\setminus F_n)=0\label{eq:CapNest}
\end{align}
in the same way of the proof of \cite[Chapter III, Theorem~2.11(i)]{MR} with the help of the 
$2$-order version of \cite[Chapter III, Proposition~1.5(iv)]{MR}. Note here that the proof of 
implication \lq\lq $\{F_n\}$ is an $\wh{\mathscr{E}}$-nest $\Longrightarrow$ \eqref{eq:CapNest}\rq\rq\; above works 
for such $\wh{h}$, though the description in \cite[Chapter III, Theorem~2.11(i)]{MR} requires that $\wh{h}$ has the shape $\wh{h}=\wh{G}_1\varphi$ for some $\wh{m}$-a.e.~strictly positive $\varphi\in L^2(\wh{X};\wh{\m})$.
 Now we consider the $\wh{h}$-transformed 
positivity preserving form $(\wh{\mathscr{E}}^{\wh{h}}, D(\wh{\mathscr{E}}^{\wh{h}}))$ on $L^2(\wh{X};\wh{h}^2\wh{\m})$ and its 
usual $2$-order capacity $\wh{\rm Cap}^{\wh{h}}_2$. Then
\begin{align*}
\wh{\rm Cap}^{\wh{h}}_2(\wh{X}\setminus F_n)
&=\inf\left\{\left. \wh{\mathscr{E}}^{\wh{h}}_2(u,u)
\,\right|\, u\in D(\wh{\mathscr{E}}^{\wh{h}}),\, u\geq1\;\wh{h}^2\wh{\m}\text{-a.e.~on }\wh{X}\setminus F_n
\right\}\\
&=\inf\left\{\left. \wh{\mathscr{E}}(u\wh{h},u\wh{h})+2\int_{\wh{X}}u^2\wh{h}^2\d\wh{\m}
\,\right|\, u\in D(\wh{\mathscr{E}}^{\wh{h}}),\, u\geq1\;\wh{\m}\text{-a.e.~on }\wh{X}\setminus F_n
\right\}\\
&=\inf\left\{\left. \wh{\mathscr{E}}_2(v,v)\,\right|\,v\in D(\wh{\mathscr{E}}),\,v\geq\wh{h}\; 
\wh{\m}\text{-a.e.~on }\wh{X}\setminus F_n
\right\}\\
&=\wh{\rm Cap}_{\wh{h},2}(\wh{X}\setminus F_n)\downarrow0\quad\text{ as }\quad n\to\infty.
\end{align*}
In the same way of the proof of $\widehat{\rm O}$kura~\cite[Theorem~4.1(4)]{Okura}, we have
\begin{align*}
\wh{\mathscr{E}}^{\wh{h}}_2(u,u)\geq \int_{\R}\mathscr{E}_2^{\wh{h}(\cdot,a)}(u(\cdot,a),u(\cdot,a))m(\d a)
\geq \int_{\R}\mathscr{E}_1^{\wh{h}(\cdot,a)}(u(\cdot,a),u(\cdot,a))m(\d a).
\end{align*}
From this, we can deduce
\begin{align*}
{\rm Cap}_{\wh{h}(\cdot,a)}(X\setminus (F_n)_a)={\rm Cap}_1^{\wh{h}(\cdot,a)}(X\setminus (F_n)_a)\to0 \quad\text{ as }\quad n\to\infty\quad\text{ for }\quad m\text{-a.e.~}a\in\R.
\end{align*}
Here ${\rm Cap}_{\wh{h}(\cdot,a)}$ is an $\wh{h}(\cdot,a)$-weighted capacity for $(\mathscr{E},D(\mathscr{E}))$.  
Therefore, $\{(F_n)_a\}$ is an $\mathscr{E}$-nest for $m$-a.e.~$a\in\R$ by applying \cite[Chapter III, Theorem~2.11(i)]{MR}. Note here that 
$\wh{h}(\cdot,a)=G_1\varphi G_1^w\phi(a)$ satisfies  \cite[the condition above in Chapter III, Theorem~2.11(i)]{MR}.
\end{proof} 

\subsection{Proof of \eqref{eq:LittlewoodPaleyStein1} for 
$f\in D(\Delta)\cap L^p(X;\m)\cap \mathscr{B}^*(X)$ under  $p\in]1,2[$}\label{subsec:3.2}

In this subsection, we return to the proof of the upper estimate \eqref{eq:LittlewoodPaleyStein1} in Theorem~\ref{thm:main1} in the case of $1<p<2$.
Recall the $\widehat{\m}$-symmetric diffusion process $\widehat{\bf X}=(\widehat{X}_t,\P_{\hat{x}})$ on $\widehat{X}:=X\times\R$ with 
$\widehat{X}_t:=(X_t, B_t)$.  
We need the following identity for our later use. 
See Shigekawa~\cite[Proposition~3.10]{ShigekawaText} for the proof.

\begin{lem}[{{\bf Shigekawa~\cite[(6.13)]{Shigekawa1},\cite[Proposition~3.10]{ShigekawaText}}}]\label{lem:ShigekawaIdntity}
Let $j:X\times[0,+\infty[\to[0,+\infty[$ be a measurable function. Then 
\begin{align}
\E_{\m\otimes\delta_a}\left[\int_0^{\tau}j(\widehat{X}_s)\d s \right]=\int_X\m(\d x)\int_0^{\infty}(a\land t)j(x,t)\d t\label{eq:Shigekawa3.10}
\end{align}
for $a\geq0$. Here $j$ is understood to be a function on $\wh{X}_{\wh{\partial}}$ with $j(\wh{\partial})=0$. 
\end{lem} 

Since $\{X_t\}_{t\geq0}$ and $\{B_t\}_{t\geq0}$ are mutually independent under $\E_{\m\otimes\delta_a}$ 
and $\m$ is the invariant measure of $\{X_t\}_{t\geq0}$, we can see the following identity for any 
$h\in \mathscr{B}_b(X)$:
\begin{align}
\E_{\m\otimes\delta_a}[h(X_{\tau})]=\int_X h(x)\m(\d x)\label{eq:invariant}
\end{align}
under 
$h(\partial)=0$ (see \cite[(3.47)]{ShigekawaText} for the proof. In \cite{ShigekawaText}, $\m\in\mathscr{P}(X)$ is assumed, but 
\cite[Proposition~3.10]{ShigekawaText} remains valid for general $\sigma$-finite $\m$). 

Now, we consider the case  $v(x,a):=u(x,|a|):=q_{|a|}^{(\alpha)}f(x)$ for 
$f\in D(\Delta)\cap L^p(X;\m)\cap \mathscr{B}^*(X)$ and $\alpha>0$. 
Applying Lemma~\ref{lem:Domain} with $\kappa=0$, $v\in D(\widehat{\mathscr{E}})$. 
It should be noted here that we still assume ${\sf BE}_2(\kappa,\infty)$, not ${\sf BE}_2(0,\infty)$. 
We abbreviate $M_{t\land \tau}^{[v]}$ as $M_t$ for simplicity.
Combining Proposition~\ref{prop:semimartingale} and Lemma~\ref{lem:Capacity}(2) with $\kappa=0$ and $\alpha>0$, there exists a non-negative sequence $\{a_n\}_{n\in\mathbb{N}}$ 
such that $\lim_{n\to\infty}a_n=+\infty$, \eqref{eq:semimartingale1} and \eqref{eq:semimartingale2} hold under $\P_{\m\otimes\delta_{a_n}}$ for any $n\in\mathbb{N}$. 

We set $V_t:=\tilde{v}(\widehat{X}_{t\land\tau})$. We apply It\^o's formula to $V_t^2$. Proposition~\ref{prop:semimartingale} implies 
\begin{align}
\d (V_t)^2&=2V_t\d M_t+2\alpha V_t^2\d t+ \d\langle M\rangle_t\notag\\
&= 2V_t\d M_t+2(g_f(\widehat{X}_t)^2+\alpha V_t^2)\d t.\label{eq:VSquare}
\end{align}
Let $\eps>0$. By applying It\^o's formula to 
$(V_t^2+\eps)^{p/2}$ again, we also have
\begin{align*}
\d(V_t^2+\eps)^{\frac{p}{2}}&=p(V_t^2+\eps)^{\frac{p}{2}-1}V_t\d M_t\\
&\hspace{1cm}+p(V_t^2+\eps)^{\frac{p}{2}-1}\left(g_f(\widehat{X}_t)^2+\alpha V_t^2 \right)\d t\\
&\hspace{2cm}+\frac{p(p-2)}{2}(V_t^2+\eps)^{\frac{p}{2}-2}V_t^2\d\langle M\rangle_t\\
&\geq p(V_t^2+\eps)^{\frac{p}{2}-1} V_t\d M_t+p(p-1)(V_t^2+\eps)^{\frac{p}{2}-1}g_f(\widehat{X}_t)^2\d t,
\end{align*}
where we used $p<2$ for the last line.

Hence, by taking the expectation of the inequality above and using 
$u(x,a)=v(x,a)$ for $a\geq0$, we have 
\begin{align}
\E_{\m\otimes\delta_a}\left[p(p-1)\int_0^{\tau}(V_t^2+\eps)^{\frac{p}{2}-1}g_f(\widehat{X}_t)^2\d t \right]&\leq 
\E_{\m\otimes\delta_a}\left[(V_{\tau}^2+\eps)^{\frac{p}{2}}-(V_0^2+\eps)^{\frac{p}{2}} \right]\notag\\
&\leq \E_{\m\otimes\delta_a}\left[(V_{\tau}^2+\eps)^{\frac{p}{2}}\right]\notag\\
&=\E_{\m\otimes\delta_a}\left[\left(\tilde{u}(\widehat{X}_{\tau})^2+\eps\right)^{\frac{p}{2}} \right]\label{eq:VarVar}\\
&=\E_{\m\otimes\delta_a}\left[\left(f(X_{\tau})^2+\eps\right)^{\frac{p}{2}} \right]\notag\\&=
\int_X(|f(x)|^2+\eps)^{\frac{p}{2}}\m(\d x),\notag
\end{align}
where we used \eqref{eq:invariant} for the last line. Here, by recalling \eqref{eq:Shigekawa3.10}, the left hand side of \eqref{eq:VarVar} is equal to 
\begin{align*}
p(p-1)\int_X\m(\d x)\int_0^{\infty}(t\land a_n)(u(x,t)^2+\eps)^{\frac{p}{2}-1}g_f(x,t)^2\d t.
\end{align*}
Therefore, by letting $\eps\to0$ and $n\to\infty$, we have 
\begin{align}
p(p-1)\int_X\m(\d x)\int_0^{\infty}t u(x,t)^{p-2}g_f(x,t)^2\d t\leq\int_X|f(x)|^p\m(\d x).\label{eq:Desired}
\end{align}

Now, we recall the maximal ergodic inequality (see Shigekawa~\cite[Theorem~3.3]{ShigekawaText} for details)
\begin{align}
\left\| \sup_{t\geq0}|P_tf|\right\|_{L^p(X;\m)}\leq \frac{p}{p-1}\|f\|_{L^p(X;\m)},\qquad p>1.\label{eq:Desired*}
\end{align}
It leads us to
\begin{align}
\|G_f\|_{L^p(X;\m)}^p&=\int_X\m(\d x)\left\{\int_0^{\infty}t|u(x,t)|^{2-p}
|u(x,t)|^{p-2}g_f(x,t)^2\d t\right\}^{\frac{p}{2}}\notag \\
&\leq \int_X\m(\d x)\left\{\int_0^{\infty}t\left(\sup_{t\geq0}|P_tf(x)| \right)^{2-p}
|u(x,t)|^{p-2}g_f(x,t)^2\d t\right\}^{\frac{p}{2}}\notag\\
&\leq \left\{\int_X\left(\sup_{t\geq0}|P_tf(x)| \right)^{p}\m(\d x) \right\}^{\frac{2-p}{2}}\notag\\
&\hspace{1cm}\times
\left\{\int_X\int_0^{\infty}t|u(x,t)|^{p-2}g_f(x,t)^2\d t\m(\d x) \right\}^{\frac{p}{2}}\label{eq:UpperG} \\
&\hspace{-0.6cm}\stackrel{\eqref{eq:Desired},\eqref{eq:Desired*}}{\leq}\frac{p^{\frac{p(1-p)}{2}}}{(p-1)^{\frac{p(3-p)}{2}}}
\left\{\int_X|f(x)|^p\m(\d x) \right\}^{\frac{2-p}{2}}
\left\{\int_X|f(x)|^p\m(\d x) \right\}^{\frac{p}{2}}\notag\\
&=\frac{p^{\frac{p(1-p)}{2}}}{(p-1)^{\frac{p(3-p)}{2}}}
\|f\|_{L^p(X;\m)}^p.\notag
\end{align}
Thus \eqref{eq:LittlewoodPaleyStein1} holds for $f\in D(\Delta)\cap L^p(X;\m)\cap \mathscr{B}^*(X)$ under $p\in]1,2[$ and $\alpha>0$. 

Next we prove that \eqref{eq:LittlewoodPaleyStein1} holds 
under $p\in]1,2[$, $\alpha=0$ and $f\in D(\Delta)\cap L^p(X;\m)\cap \mathscr{B}^*(X)$. 
We note that $\Gamma(Q_{\cdot}^{(\alpha)}f)^{\frac12}\to \Gamma(Q_{\cdot}^{(0)}f)^{\frac12}$ in 
$L^2(X\times[0,\infty[:\m\otimes e^{-t}\d t)$ as $\alpha\to0$ for $f\in D(\mathscr{E})$. 
Indeed, 
\begin{align*}
\int_0^{\infty}e^{-t}&\d t\int_X\left|\Gamma(Q_t^{(\alpha)}f)^{\frac12}-\Gamma(Q_t^{(0)}f)^{\frac12} \right|^2\d\m\\
&\leq\int_0^{\infty}e^{-t}\d t\int_X\Gamma(Q_t^{(\alpha)}f-Q_t^{(0)}f)\d\m\\
&=\int_0^{\infty}e^{-t}\d t\int_0^{\infty}
\lambda\left( e^{-\sqrt{\alpha+\lambda}t}-e^{-\sqrt{\lambda}t}\right)^2\d(E_{\lambda}f,f)\\
&=\int_0^{\infty}\lambda\left[\frac{1}{2\sqrt{\alpha+\lambda}+1}-\frac{2}{\sqrt{\alpha+\lambda}+\sqrt{\lambda}+1}+\frac{1}{2\sqrt{\lambda}+1} \right]
\d(E_{\lambda}f,f)\\
&\to 0\quad \text{ as }\quad \alpha\downarrow0.
\end{align*}
Moreover, 
$\sqrt{\alpha-\Delta}Q_{\cdot}^{(\alpha)}f\to \sqrt{-\Delta}Q_{\cdot}^{(0)}f$ in $L^2(X\times[0,\infty[;\m\otimes e^{-t}\d t)$ as $\alpha\to0$ for $f\in D(\mathscr{E})$. Indeed, 
\begin{align*}
\int_0^{\infty}&e^{-t}\d t\int_X\left|\sqrt{\alpha-\Delta}Q_t^{(\alpha)}f-\sqrt{-\Delta}Q_t^{(0)}f\right|^2\d\m\\
&=\int_0^{\infty}e^{-t}\d t\int_0^{\infty}
\left|\sqrt{\alpha+\lambda}e^{-\sqrt{\alpha+\lambda}t}-\sqrt{\lambda}e^{-\sqrt{\lambda}t} \right|^2\d(E_{\lambda}f,f)\\
&=\int_0^{\infty}\left(\frac{\alpha+\lambda}{1+2\sqrt{\alpha+\lambda}}-\frac{2\sqrt{\alpha+\lambda}\sqrt{\lambda}}{1+\sqrt{\alpha+\lambda}+\sqrt{\lambda}}+\frac{\lambda}{1+2\sqrt{\lambda}} \right)\d(E_{\lambda}f,f)
&\to 0\quad \text{ as }\quad \alpha\downarrow0.
\end{align*}
Then there exists a subsequence $\{\alpha_k\}$ tending to $0$ as $k\to\infty$ such that $\Gamma(Q_t^{(\alpha_k)}f)(x)\to \Gamma(Q_t^{(0)}f)(x)$ and $\sqrt{\alpha_k-\Delta}Q_t^{(\alpha_k)}f(x) \to \sqrt{-\Delta}Q_t^{(0)}f(x)$ as $k\to\infty$ $\widehat{\m}$-a.e.~$(x,t)$.   

Therefore, we can conclude that \eqref{eq:LittlewoodPaleyStein1} holds 
under $p\in]1,2[$, $\alpha=0$ and $f\in D(\Delta^{\kappa^+})\cap L^p(X;\m)\cap \mathscr{B}^*(X)$ by way of Fatou's lemma and 
the estimate \eqref{eq:UpperG} under $\alpha>0$.

\subsection{Proof of \eqref{eq:LittlewoodPaleyStein1} for $f\in D(\Delta)\cap L^p(X;\m)\cap \mathscr{B}^*(X)$ under  $p\in]2,+\infty[$}\label{subsec:p>2}
In this subsection, we assume Assumption~\ref{asmp:Tamed}, consequently, 
the estimate \eqref{eq:gradCont} holds for $f\in D(\mathscr{E})$.  
We still assume $\kappa\in S_D({\bf X})$ and $\kappa^-\in S_{E\!K}({\bf X})$. 
In the case of $p>2$, we need additional functions, namely $H$-functions
\begin{align*}
{H_f^{\rightarrow}}^{\kappa}:&=\left\{\int_0^{\infty}tQ_t^{(\alpha),\kappa}(g_f^{\rightarrow}(\cdot,t)^2)(x)\d t \right\}^{\frac12},\\
{H_f^{\uparrow}}^{\kappa}:&=\left\{\int_0^{\infty}tQ_t^{(\alpha),\kappa}(g_f^{\uparrow}(\cdot,t)^2)(x)\d t \right\}^{\frac12},\\
{H_f}^{\kappa}:&=\left\{\int_0^{\infty}tQ_t^{(\alpha),\kappa}(g_f(\cdot,t)^2)(x)\d t \right\}^{\frac12}
\end{align*}
under $\alpha\geq C_{\kappa}$. 
We write ${H_f^{\rightarrow}}:={H_f^{\rightarrow}}^0$, 
${H_f^{\uparrow}}:={H_f^{\uparrow}}^0$ and ${H_f}:={H_f}^0$. 
Note that ${H_f^{\rightarrow}}^{\kappa}$, ${H_f^{\uparrow}}^{\kappa}$ and 
${H_f}^{\kappa}$ (resp.~${H_f^{\rightarrow}}$, ${H_f^{\uparrow}}$ and 
${H_f}$) depend on $\alpha\geq C_{\kappa}$ (resp.~$\alpha\geq0$ if $\kappa^-=0$). 

\medskip

Not only Lemma~\ref{lem:ShigekawaIdntity}, we need the following inequality extending  
\cite[Proposition~3.11]{ShigekawaText}. 
\begin{lem}\label{lem:ShigekawaInequality}
Assume $\kappa^+\in S_D({\bf X})$ and $\kappa^-\in S_{E\!K}({\bf X})$.  
Let $j:X\times[0,+\infty[\to[0,+\infty[$ be a measurable function. Then 
\begin{align}
\E_{\m\otimes\delta_a}\left[\left. \int_0^{\tau}j(\widehat{X}_s)\d s\,\right|\, X_{\tau}\right]
\geq \int_0^{\infty}(a\land t)
Q_t^{(\alpha),\kappa}(j(\cdot,t))(X_{\tau})
\d t\label{eq:ShigekawaProp3.11}
\end{align}
holds for $\alpha\geq C_{\kappa}$. Here $j$ is understood to be a function on $\wh{X}_{\wh{\partial}}$ with $j(\wh{\partial})=0$. 
\end{lem}
\begin{remark}
{\rm Since $\alpha\geq C_{\kappa}>0$ under $\kappa^-\ne0$, we can not expect the equality in \eqref{eq:ShigekawaProp3.11} under $\alpha=0$.  
}
\end{remark}
\begin{proof}[\bf Proof of Lemma~{\boldmath\ref{lem:ShigekawaInequality}}]
We may assume $\alpha>C_{\kappa}$, because the case $\alpha=C_{\kappa}$ in 
\eqref{eq:ShigekawaProp3.11} 
can be deduced by the limit $\alpha\downarrow C_{\kappa}$.  By taking an increasing approximating sequence $\{j_n\}\subset L^1(X\times[0,+\infty[;\widehat{\m})$ of non-negative measurable functions to $j$, we may assume 
$j\in L^1(X\times[0,+\infty[;\widehat{\m})$. 
Thanks to the transience of $\widehat{\bf X}_{\widehat{G}}$ with the integrability of $j$, we have 
\begin{align}
\E_{(x,a)}\left[\int_0^{\tau}j(\widehat{X}_s)\d s \right]<\infty,\quad \widehat{\m}\text{-a.e.~}(x,a)\label{eq:finiteness}
\end{align} 
by \cite[(1.5.4)]{FOT}. 
To prove \eqref{eq:ShigekawaProp3.11}, it suffices to show 
\begin{align}
\E_{(x,a)}\left[f(X_{\tau})\int_0^{\tau}j(\widehat{X}_s)\d s  \right]\geq
\E_{(x,a)}\left[\int_0^{\tau} 
\tilde{v}(\widehat{X}_t)j(\widehat{X}_t)\d t \right]\quad \widehat{\m}\text{-a.e.~}(x,a).
\label{eq:Shigekawa2.46}
\end{align}
for $v(x,a):=q_{|a|}^{(\alpha),\kappa}f(x)$ with 
$f\in \mathscr{B}_+(X)$ and $\alpha>C_{\kappa}$.
Indeed, by taking an integration of \eqref{eq:Shigekawa2.46} with respect to $\m$, we have
\begin{align*}
\E_{\m\otimes\delta_a}\left[f(X_{\tau})\int_0^{\tau}j(\widehat{X}_s)\d s  \right]&\geq
\E_{\m\otimes\delta_a}\left[\int_0^{\tau}
\tilde{v}(\widehat{X}_t)j(\widehat{X}_t)\d t \right]
\\
&\hspace{-0.2cm}\stackrel{\eqref{eq:Shigekawa3.10}}{=}\int_X\m(\d x)\int_0^{\infty}(a\land t)v(x,t)j(x,t)\d t\\
&=\int_0^{\infty}(a\land t)\d t\int_X
Q_t^{(\alpha),\kappa}f(x)
j(x,t)\m(\d x)\\
&=\int_0^{\infty}(a\land t)\d t\int_X Q_t^{(\alpha),\kappa}(j(\cdot,t))(x)f(x)\m(\d x)
\quad\text{(symmetry)}\\
&=\int_0^{\infty}(a\land t)\E_{\m\otimes\delta_a}\left[Q_t^{(\alpha),\kappa}
(j(\cdot,t))(X_{\tau})
f(X_{\tau}) \right]\d t,
\end{align*}
which implies \eqref{eq:ShigekawaProp3.11}.  
Here we apply \eqref{eq:invariant} in the last equality.

From now on, we prove  \eqref{eq:Shigekawa2.46} for $f\in D(\Delta^{\kappa})\cap\mathscr{B}^*_+(X)$. 
Fix such an $f$ and set $v(x,a):=u(x,|a|):=q_{|a|}^{(\alpha),\kappa}f(x)$ with $\alpha>C_{\kappa}$. 
By Lemma~\ref{lem:Domain}, $v\in D(\widehat{\mathscr{E}}^{\,\wh{\kappa}})=D(\widehat{\mathscr{E}})$ for $\alpha>C_{\kappa}$.  As proved in Proposition~\ref{prop:semimartingale}, we have 
the semi-martingale decomposition \eqref{eq:semimartingale1} with \eqref{eq:semimartingale2} and \eqref{eq:semimartingale3}.  We set 
\begin{align*}
M_f(t):&=\wh{M}_{t\land\tau}^{[v]}+\tilde{v}(\widehat{X}_0)\\
&=\tilde{v}(\widehat{X}_{t\land\tau})-\alpha\int_0^{t\land\tau}\tilde{v}(\widehat{X}_s)\d s,\quad t\geq0\quad\P_{\hat{x}}\text{-a.s.~for q.e.~}\hat{x}.
\end{align*}
Then by \eqref{eq:semimartingale3}, $M_f(t)$ is a uniformly integrable  
martingale with respect to $\P_{(x,a)}$ for q.e.~$(x,a)$. 
As $t\to\infty$, $M_f(t)$ converges to 
\begin{align*}
\tilde{v}(\widehat{X}_{\tau})-\alpha\int_0^{\tau}\tilde{v}(\widehat{X}_s)\d s=f(X_{\tau})-\alpha\int_0^{\tau}\tilde{v}(\widehat{X}_s)\d s.
\end{align*}
As a consequence, $M_f(t)$ is represented by 
\begin{align}
M_f(t)=\E_{(x,a)}\left[\left. f(X_{\tau})-\alpha\int_0^{\tau}\tilde{v}(\widehat{X}_s)\d s \,\right|\,\mathcal{F}_t\right],\quad \P_{(x,a)}\text{-a.s.~for q.e.~}(x,a).\label{eq:closedmartingale}
\end{align}
Thus 
\begin{align*}
\E_{(x,a)}\left[f(X_{\tau})\int_0^{\tau}j(\widehat{X}_s)\d s  \right]&=
\int_0^{\infty}\E_{(x,a)}\left[j(\widehat{X}_t) f(X_{\tau})\1_{\{t\leq\tau\}} \right]\d t\\
&=\int_0^{\infty}\E_{(x,a)}\left[\E_{(x,a)}[f(X_{\tau})\,|\,\mathcal{F}_t]j(\widehat{X}_t)\1_{\{t\leq\tau\}} \right]\d t\\
&\hspace{-0.2cm}\stackrel{\eqref{eq:closedmartingale}}{=}\int_0^{\infty}\E_{(x,a)}\left[\tilde{v}(\widehat{X}_t)j(\widehat{X}_t)\1_{\{t\leq\tau\}} \right]\d t\\
&\hspace{1cm}+\alpha\int_0^{\infty}\E_{(x,a)}\left[\E_{(x,a)}\left[\left.\int_{t\land\tau}^{\tau} \tilde{v}(\widehat{X}_s)\d s\,\right|\,\mathcal{F}_t \right]j(\widehat{X}_t)\1_{\{t\leq\tau\}} \right]\d t\\
&\geq \int_0^{\infty}\E_{(x,a)}\left[\tilde{v}(\widehat{X}_t)j(\widehat{X}_t)\1_{\{t\leq\tau\}} \right]\d t\\
&=\E_{(x,a)}\left[\int_0^{\tau}\tilde{v}(\widehat{X}_t)j(\widehat{X}_t)\d t \right].
\end{align*}
This shows that \eqref{eq:Shigekawa2.46} holds for $f\in D(\Delta^{\kappa})\cap\mathscr{B}_+^*(X)$ with 
$\alpha>C_{\kappa}$. 

Next we prove that \eqref{eq:Shigekawa2.46} holds $f\in L^2(X;\m)\cap \mathscr{B}_b(X)_+$ with 
$\alpha>C_{\kappa}$. For such an $f$, we set 
$f_n:=p_{1/n}^{\kappa}f$. Then $\{f_n\}$ is uniformly bounded and $f_n\in D(\Delta^{\kappa})\cap 
\mathscr{B}^*_b(X)_+$. We already prove that  \eqref{eq:Shigekawa2.46} holds for $f_n$.  
Letting $n\to\infty$ with \eqref{eq:finiteness}, 
we can conclude that \eqref{eq:Shigekawa2.46} holds for $f\in L^2(X;\m)\cap \mathscr{B}_b(X)_+$ by way of Lebesgue's dominated convergence theorem. Finally, approximating $f\in \mathscr{B}(X)_+$ by an increasing sequence in $L^2(X;\m)\cap \mathscr{B}_b(X)_+$, 
we see that  \eqref{eq:Shigekawa2.46} still holds for any $f\in \mathscr{B}(X)_+$ with $\alpha>C_{\kappa}$.
\end{proof}

We begin by the following proposition: 

\begin{prop}\label{prop:ShigekawaYoshida}
For $p>2$, the following inequality holds for any $f\in D(\Delta)\cap L^p(X;\m)\cap\mathscr{B}^*(X)$ with $\alpha\geq C_{\kappa}$ and $\alpha>0$: 
\begin{align*}
\|{H_f}^{\kappa}\|_{L^p(X;\m)}\lesssim \|f\|_{L^p(X;\m)}.
\end{align*}
\end{prop}
\begin{proof}[\bf Proof]  
By a slight modification, we can prove in the same way as the proof of Shigekawa-Yoshida~\cite[Proposition~4.2]{ShigekawaYoshida}. However we give the proof for readers' convenience.

Let us recall that $v(x,a):=q_{|a|}^{(\alpha)}f(x)$ with $f\in D(\Delta)$ satisfies $v\in D(\widehat{\mathscr{E}})$ under $\alpha>0$, 
due to \eqref{eq:VSquare}, we have 
\begin{align}
V_{t\land\tau}^2-V_0^2=2\int_0^{t\land\tau}V_s\d M_s+2\int_0^{t\land\tau}
\left(\alpha V_s^2+g_f(\widehat{X}_s)^2 \right)\d s.\label{eq:VSquareStochas}
\end{align}
Since $A_t:=2\int_0^{t\land\tau}
\left(\alpha V_s^2+g_f(\widehat{X}_s)^2 \right)\d s$, $t\geq0$, is a continuous increasing process, 
\eqref{eq:VSquareStochas} implies that $Z_t:=V_{t\land\tau}^2-V_0^2$, $t\geq0$, is a submartingale. 

Now we need an inequality for submartingale. Let $\{Z_t\}_{t\geq0}$ be a continuous submartingale with the Doob-Meyer decomposition 
$Z_t=M_t+A_t$, where $\{M_t\}_{t\geq0}$ is a continuous martingale and $\{A_t\}_{t\geq0}$ is a continuous increasing process with $A_0=0$. Due to Lenglart-L\'epingle-Pratelli~\cite{LenglartLepinglePratelli}, it holds that 
\begin{align}
\E[A_{\infty}^p]\leq(2p)^p\E\left[\sup_{t\geq0}|Z_t|^p \right], \quad p>1.\label{eq:LLP}
\end{align} 
Then by using \eqref{eq:LLP} and Doob's inequality, we have 
\begin{align}
\E_{\m\otimes\delta_{a_n}}\left[\left\{2\int_0^{\tau}\left(\alpha V_s^2+g_f(\widehat{X}_s)^2 \right)\d s \right\}^{\frac{p}{2}} \right]
&\lesssim\E_{\m\otimes\delta_{a_n}}\left[\sup_{t\geq0}|V_{t\land\tau}^2-V_0^2|^{\frac{p}{2}} \right]\notag\\
&\lesssim\E_{\m\otimes\delta_{a_n}}\left[|V_{\tau}^2-V_0^2|^{\frac{p}{2}} \right]\notag\\
&=\E_{\m\otimes\delta_{a_n}}\left[|{v}(\widehat{X}_{\tau})^2-{v}(\widehat{X}_0)^2|^{\frac{p}{2}} \right]\label{eq:Doob}\\
&=\E_{\m\otimes\delta_{a_n}}
\left[|(Q_0^{(\alpha)}
f(X_{\tau}))^2-
(Q_{a_n}^{(\alpha)}f(X_0))^2|^{\frac{p}{2}}
 \right]\notag\\
&\lesssim \E_{\m\otimes\delta_{a_n}}
\left[|Q_0^{(\alpha)}f(X_{\tau})|^p \right]+
\E_{\m\otimes\delta_{a_n}}\left[|Q_{a_n}^{(\alpha)}f(X_0)|^p \right]\notag\\
&\hspace{-0.2cm}\stackrel{\eqref{eq:invariant}}{=}
\|f\|_{L^p(X;\m)}^p+\|Q_{a_n}^{(\alpha)}f\|_{L^p(X;\m)}^p
\lesssim\|f\|_{L^p(X;\m)}^p.\notag
\end{align}
On the other hand, by using \eqref{eq:ShigekawaProp3.11}, \eqref{eq:Doob} and Jensen's inequality, we have 
\begin{align*}
\|{H_f}^{\!\!\kappa}\|_{L^p(X;\m)}^p&=\left\|\left\{\int_0^{\infty}t Q_t^{(\alpha),\kappa}(g_f(\cdot,t)^2)\d t 
\right\}^{\frac{p}{2}} \right\|_{L^1(X;\m)}\\
&=\lim_{n\to\infty}\left\|
\left\{\int_0^{\infty}(a_n\land t) Q_t^{(\alpha),\kappa}(g_f(\cdot,t)^2)\d t 
\right\}^{\frac{p}{2}}
 \right\|_{L^1(X;\m)}\\
&\hspace{-0.2cm}\stackrel{\eqref{eq:invariant}}{=}
\lim_{n\to\infty}\E_{\m\otimes\delta_{a_n}}\left[\left\{\int_0^{\infty}(a_n\land t)Q_t^{(\alpha),\kappa}(g_f(\cdot,t)^2)(X_{\tau})\d t \right\}^{\frac{p}{2}} \right] \notag\\
 &\hspace{-0.2cm}\stackrel{\eqref{eq:ShigekawaProp3.11}}{\leq}
 \varliminf_{n\to\infty}
 \E_{\m\otimes\delta_{a_n}}
 \left[\E_{\m\otimes\delta_{a_n}}
 \left[\left.
\int_0^{\tau}g_f(\widehat{X}_s)^2\d s
\,\right|\,X_{\tau}
\right]^{\frac{p}{2}}
\right]\\
&\leq\varliminf_{n\to\infty}\E_{\m\otimes\delta_{a_n}}\left[\E_{\m\otimes\delta_{a_n}}
\left[\left.\left(\int_0^{\tau}g_f(\widehat{X}_s)^2\d s \right)^{\frac{p}{2}} \,\right|\, X_{\tau} 
\right] \right]\\
&=\varliminf_{n\to\infty}\E_{\m\otimes\delta_{a_n}}\left[
\left(\int_0^{\tau}g_f(\widehat{X}_s)^2\d s \right)^{\frac{p}{2}}\right]\\
&\leq\varliminf_{n\to\infty}\E_{\m\otimes\delta_{a_n}}\left[
\left\{\int_0^{\tau}\left(\alpha V_s^2+g_f(\widehat{X}_s)^2 \right)\d s
 \right\}^{\frac{p}{2}}
\right]\\
&\hspace{-0.2cm}\stackrel{\eqref{eq:Doob}}{\lesssim} \|f\|_{L^p(X;\m)}^p.
\end{align*}
In applying \eqref{eq:ShigekawaProp3.11}, we use $\alpha\geq C_{\kappa}$. 
This completes the proof. 
\end{proof} 

Next we study the relationship between $G$-functions and $H$-functions. In the proof of this proposition, Assumption~\ref{asmp:Tamed} plays a key role. 

\begin{prop}\label{prop:GHEest}
\begin{enumerate}
\item[{\rm (1)}] For any $f\in D(\mathscr{E})$ and $\alpha\geq C_{\kappa}$, the following inequality 
holds: 
\begin{align*}
G_f^{\uparrow}\leq 2\sqrt{C}{H_f^{\uparrow}}^{\kappa},\quad\m\text{-a.e.}
\end{align*} 
\item[{\rm (2)}] For any $f\in D(\mathscr{E})$ and $\alpha\geq0$, the following inequality 
holds: 
\begin{align*}
G_f^{\rightarrow}\leq 2{H_f^{\rightarrow}}.
\end{align*} 
\end{enumerate}
\end{prop}
\begin{proof}[\bf Proof]  
First we show (1). 
For this, we prove 
\begin{align}
\Gamma(Q_t^{(\alpha)}f)^{\frac12}\leq\int_0^{\infty}e^{-\alpha s}\Gamma(P_sf)^{\frac12}\lambda_t(\d s),\quad \m\text{-a.e.}\label{eq:Bochner}
\end{align}
We set a $D(\mathscr{E})$-valued simple function $f_n(s)$ by 
\begin{align*}
f_n(s):=\sum_{k=0}^{2^n\cdot n-1}\1_{[\frac{k}{2^n},\frac{k+1}{2^n}[}(s)P_{\frac{k}{2^n}}f+\1_{[n,\infty[}(s)P_nf.
\end{align*}
Then for each $s\geq0$, $\{f_n(s)\}$ converges to $P_sf$ in $(\mathscr{E},D(\mathscr{E}))$ as $n\to\infty$, consequently 
$\Gamma(f_n(s))^{\frac12}$ converges to $\Gamma(P_sf)^{\frac12}$ in $L^2(X;\m)$ as $n\to\infty$. 
Similarly, setting $Q_t f_n:=\int_0^{\infty}e^{-\alpha s}f_n(s)\lambda_t(\d s)$, 
$\{Q_tf_n\}$  converges to $Q_t^{(\alpha)}f$ in 
 $(\mathscr{E},D(\mathscr{E}))$ as $n\to\infty$, and $\Gamma(Q_t f_n)^{\frac12}$ converges to $\Gamma(Q_t^{(\alpha)}f)^{\frac12}$ in $L^2(X;\m)$ as $n\to\infty$. Since 
\begin{align*}
\Gamma(f_n(s))^{\frac12}=\left\{\begin{array}{cl}\Gamma(P_{\frac{k}{2^n}}f)^{\frac12}, & \quad\text{ if }\quad  s\in[\frac{k}{2^n},\frac{k+1}{2^n}[,\quad k\in\{0,1,\cdots, 2^n\cdot n-1\},\\  \Gamma(P_nf)^{\frac12}, & \quad\text{ if }\quad s\geq n,\end{array}\right.
\end{align*} 
we can see 
\begin{align*}
\Gamma(Q_tf_n)^{\frac12}\leq \int_0^{\infty}e^{-\alpha s}\Gamma(f_n(s))^{\frac12}\lambda_t(\d s).
\end{align*} 
Letting $n\to\infty$, we obtain \eqref{eq:Bochner}.  
By \eqref{eq:Bochner}, the upper estimate \eqref{eq:gradCont} and Schwarz's inequality, we have the following estimate for any $\alpha\geq C_{\kappa}$ and $f\in D(\mathscr{E})$:
\begin{align}
\Gamma(Q_t^{(\alpha)}f)^{\frac12}&\leq\int_0^{\infty}e^{-\alpha s}\Gamma(P_sf)^{\frac12}\lambda_t(\d s)\notag\\
&\leq\int_0^{\infty}e^{-\alpha s}P_s^{\kappa}\Gamma(f)^{\frac12}\lambda_t(\d s)\label{eq:BEQ1}\\
&=Q_t^{(\alpha),\kappa}\Gamma(f)^{\frac12}=q_t^{(\alpha),\kappa}\Gamma^*(f)^{\frac12},\notag
\end{align}
where $\Gamma^*(f)$ is a Borel $\m$-version of $\Gamma(f)$. 
Then \eqref{eq:BEQ1} yields that for each $t\geq0$
\begin{align}
g_f^{\uparrow}(x,2t)^2&=\Gamma(Q_{2t}^{(\alpha)}f)(x)\notag\\
&=\Gamma\left(Q_t^{(\alpha)}(Q_t^{(\alpha)}f)\right)(x)\notag\\
&\leq \left(q_t^{(\alpha),\kappa}\Gamma^*(Q_t^{(\alpha)}f)^{\frac12}(x)\right)^2 \notag\\
&\leq q_t^{(\alpha),\kappa}1(x) \cdot Q_t^{(\alpha),\kappa}\Gamma(Q_t^{(\alpha)}f)(x)\label{eq:BEQ2}\\
&= \left(\int_0^{\infty}e^{-\alpha s}p_s^{\kappa}1(x)\lambda_t(\d s) \right)
Q_t^{(\alpha),\kappa}(g_f^{\uparrow}(\cdot,t)^2)(x)
\notag\\
&\hspace{-0.2cm}\stackrel{\eqref{eq:Contra}}{\leq} Ce^{-\sqrt{\alpha-C_{\kappa}}t}Q_t^{(\alpha),\kappa}(g_f^{\uparrow}(\cdot,t)^2)(x),\quad \m\text{-a.e.~}x\in X.\notag
\end{align}
Since $t\mapsto Q_t^{(\alpha),\kappa}f$ is a $D(\mathscr{E})$-valued continuous function, 
$t\mapsto g_f^{\uparrow}(x,2t)^2=\Gamma(Q_{2t}^{(\alpha)}f)$ is an $L^1(X;\m)$-valued continuous function, moreover, 
$t\mapsto Q_t^{(\alpha),\kappa}(g_f^{\uparrow}(\cdot,2t)^2)$ is an $L^1(X;\m)$-valued continuous function. 
Thus, $g_f^{\uparrow}(x,2t)^2\leq  Q_t^{(\alpha),\kappa}(g_f^{\uparrow}(\cdot,t)^2)(x)$ for all $t\geq0$ $\m$-a.e.~$x\in X$.

Therefore we have 
\begin{align*}
\left(G_f^{\uparrow}(x)\right)^2&=4\int_0^{\infty}tg_f^{\uparrow}(x,2t)^2\d t\\
&\leq 4C\int_0^{\infty}t Q_t^{(\alpha),\kappa}(g_f^{\uparrow}(\cdot,t)^2)(x)\d t=
4C({H_f^{\uparrow}}^{\kappa}(x))^2,\quad \m\text{-a.e.~}x\in X,
\end{align*}
where we changed the variable $t$ in $2t$ in the first line and used \eqref{eq:BEQ2} for the second line. 
Next we show (2). Using the semigroup property of $(Q_t^{(\alpha)})_{t\geq0}$ under $\alpha\geq0$,  we have 
\begin{align*}
Q_{t+s}^{(\alpha)}f=Q_t^{(\alpha)}Q_s^{(\alpha)}f.
\end{align*}
Differentiating with respect to $s$ and setting $s=t$, we have 
\begin{align*}
\left.\frac{\partial}{\partial a}Q_a^{(\alpha)}f\right|_{a=2t} =
Q_t^{(\alpha)}\left.\frac{\partial}{\partial a}Q_a^{(\alpha)}f\right|_{a=t}.
\end{align*}
Therefore, from the sub-Markovian property of $(Q_t^{(\alpha)})_{t\geq0}$
\begin{align*}
g_f^{\rightarrow}(x,2t)^2&=\left(\left.\frac{\partial}{\partial a}Q_a^{(\alpha)}f(x)\right|_{a=2t} \right)^2
=\left(Q_t^{(\alpha)}\left.\frac{\partial}{\partial a}Q_a^{(\alpha)}f(x)\right|_{a=t} \right)^2\\&
\leq Q_t^{(\alpha)}\left(\left|\frac{\partial}{\partial t}Q_t^{(\alpha)}f\right|^2\right)(x)=Q_t^{(\alpha)}(g_f^{\rightarrow}(\cdot,t)^2)(x).
\end{align*}
Integrating this, we have
\begin{align*}
G_f^{\rightarrow}(x)^2=4\int_0^{\infty}tg_f^{\rightarrow}(x,2t)^2\ dt
\leq 4\int_0^{\infty}tQ_t^{(\alpha)}(g_f^{\rightarrow}(\cdot,t))(x)^2\d t=4H_f^{\rightarrow}(x)^2.
\end{align*}
This completes the proof. 
\end{proof} 

Combining Propositions~\ref{prop:ShigekawaYoshida} and \ref{prop:GHEest}, 
\eqref{eq:LittlewoodPaleyStein1} holds for 
$f\in D(\Delta)\cap L^p(X;\m)\cap \mathscr{B}^*(X)$ under $p\in]2,+\infty[$ and $\alpha\geq C_{\kappa}$ with 
$\alpha>0$.  
Under $\kappa^-=0$ and $\alpha=0$, we can see that \eqref{eq:LittlewoodPaleyStein1} holds for 
$f\in D(\Delta)\cap L^p(X;\m)\cap \mathscr{B}^*(X)$ and $p\in]2,+\infty[$ by way of the same argument 
as in the previous subsection.

\subsection{Proof of the upper estimates \eqref{eq:LittlewoodPaleyStein1} and \eqref{eq:LittlewoodPaleyStein1+} 
under $p=2$}\label{subsec:p=2}
We prove the upper estimates \eqref{eq:LittlewoodPaleyStein1} and \eqref{eq:LittlewoodPaleyStein1+} under $p=2$. 
\begin{prop}\label{prop:GHEestp=2}
When $\alpha>0$,
\begin{align}
2\|G_f^{\rightarrow}\|_{L^2(X;\m)}=\|f\|_{L^2(X;\m)},\quad \text{ for all }\quad f\in L^2(X;\m),\label{eq:EqE+1}\\
2\|G_f^{\uparrow}\|_{L^2(X;\m)}\leq\|f\|_{L^2(X;\m)},\quad \text{ for all }\quad f\in L^2(X;\m).\label{eq:EqE+2}
\end{align}
When $\alpha=0$,
\begin{align}
2\|G_f^{\rightarrow}\|_{L^2(X;\m)}=\|f-E_{o}f\|_{L^2(X;\m)},\quad \text{ for all }\quad f\in L^2(X;\m),\label{eq:EqEo1}\\
2\|G_f^{\uparrow}\|_{L^2(X;\m)}=\|f-E_{o}f\|_{L^2(X;\m)},\quad \text{ for all }\quad f\in L^2(X;\m).\label{eq:EqEo2}
\end{align}
In particular, $\sqrt{2}\|G_f\|_{L^2(X;\m)}\leq \|f\|_{L^2(X;\m)}$ for all $f\in L^2(X;\m)$ under $\alpha>0$, 
and $\sqrt{2}\|G_f\|_{L^2(X;\m)}= \|f-E_{o}f\|_{L^2(X;\m)}$ for all $f\in L^2(X;\m)$ under $\alpha=0$. 
\end{prop}
\begin{proof}[\bf Proof]  
First note that for $f\in L^2(X;\m)$
\begin{align*}
\|G_f^{\rightarrow}\|_{L^2(X;\m)}^2&=\int_X\m(\d x)\int_0^{\infty}t\left|\frac{\partial}{\partial t}Q_t^{(\alpha)}f(x) \right|^2\d t\\
&=\int_0^{\infty}t\d t \left\|\frac{\partial}{\partial t}Q_t^{(\alpha)}f \right\|_{L^2(X;\m)}^2\\
&=\int_0^{\infty} t\d t\int_0^{\infty}(\lambda+\alpha)e^{-2\sqrt{\lambda+\alpha}t}\d(E_{\lambda}f,f)\\
&=\int_0^{\infty}\left\{\int_0^{\infty}t(\lambda+\alpha)e^{-2\sqrt{\lambda+\alpha}t}\d t \right\}\d(E_{\lambda}f,f).
\end{align*}
Now, noticing that 
\begin{align*}
\int_0^{\infty}t(\lambda+\alpha)e^{-2\sqrt{\lambda+\alpha}t}\d t=\left\{\begin{array}{cl}\frac14, & \lambda+\alpha>0, \\0, & \lambda+\alpha=0,\end{array}\right.
\end{align*}
we have 
\begin{align*}
\|G_f^{\rightarrow}\|_{L^2(X;\m)}^2=\left\{\begin{array}{lc}\frac14\|f\|_{L^2(X;\m)}^2, & \alpha>0, \\ \frac14\|f-E_{o}f\|_{L^2(X;\m)}^2, & \alpha=0.\end{array}\right.
\end{align*}
Next, let us consider $G_f^{\uparrow}$: 
\begin{align*}
\|G_f^{\uparrow}\|_{L^2(X;\m)}^2&=\int_X\m(\d x)\int_0^{\infty}t\Gamma(Q_t^{(\alpha)}f)\d t\\
&=\int_0^{\infty} t\d t\int_X\Gamma(Q_t^{(\alpha)}f)\m(\d x)\\
&=\int_0^{\infty} t\d t\int_X(-\Delta Q_t^{(\alpha)})f(x)Q_t^{(\alpha)}f(x)\m(\d x)\\
&=\int_0^{\infty}t \d t\int_0^{\infty}\lambda e^{-2\sqrt{\lambda+\alpha}t}\d (E_{\lambda}f,f)\\
&=\int_{]0,+\infty[}\frac{\lambda}{4(\lambda+\alpha)}\d (E_{\lambda}f,f).
\end{align*}
Thus, $\|G_f^{\uparrow}\|_{L^2(X;\m)}^2\leq \frac14\|f\|_{L^2(X;\m)}^2$ follows when $\alpha>0$ and 
$\|G_f^{\uparrow}\|_{L^2(X;\m)}^2=\frac14\|f-E_{o}f\|_{L^2(X;\m)}^2$ follows when $\alpha=0$. 
Therefore, $\|G_f\|_{L^2(X;\m)}^2=\|G_f^{\rightarrow}\|_{L^2(X;\m)}^2+\|G_f^{\uparrow}\|_{L^2(X;\m)}^2\leq\frac12
\|f\|_{L^2(X;\m)}^2$ for $\alpha>0$ and $\|G_f\|_{L^2(X;\m)}^2=\frac12
\|f-E_{o}f\|_{L^2(X;\m)}^2$ for $\alpha=0$. 
\end{proof} 
\begin{lem}\label{lem:SpectralMeasure}
Suppose $f\in L^2(X;\m)$. Then $E_of\in D(\mathscr{E})$ and $\mathscr{E}(E_of,E_of)=0$. 
In particular, $E_of=0$ if $(\mathscr{E},D(\mathscr{E}))$ is transient. 
Moreover, %under $\kappa^-=0$ and $\alpha=0$, 
$G_{E_of}=0$ for $f\in L^2(X;\m)$. 
\end{lem}
\begin{proof}[\bf Proof]  
In view of the spectral representation, 
\begin{align*}
(P_tE_of,E_o f)&=\int_0^{\infty}e^{-\lambda t}\d (E_{\lambda}E_of,E_of)\\
&=\lim_{n\to\infty}\sum_{k=0}^{2^n\cdot n-1}e^{-\frac{k}{2^n}t}\int_{[\frac{k}{2^n},\frac{k+1}{2^n}[ }
\d(E_{\lambda}E_of,E_of)+e^{-nt}\int_{[n,+\infty[}\d(E_{\lambda}E_of,E_of)
\\
&=\lim_{n\to\infty}\sum_{k=0}^{2^n\cdot n-1}e^{-\frac{k}{2^n}t}
(E_{[\frac{k}{2^n},\frac{k+1}{2^n}[}E_{\{0\}}f,E_{\{0\}}f)
+e^{-nt}(E_{[n,+\infty[}E_{\{0\}}f,E_{\{0\}}f)
\\
&=(E_{[0,\frac{1}{2^n}[}E_{\{0\}}f,E_{\{0\}}f)
=
(E_{\{0\}}f,E_{\{0\}}f)=\|E_of\|_{L^2(X;\m)}^2,
\end{align*}
which yields the first conclusion. In the same way, we can get $
\mathscr{E}(P_sE_of,P_sE_of)
=\mathscr{E}(Q_t^{(0)}\!E_of,Q_t^{(0)}E_of)=0$ for 
$f\in L^2(X;\m)$. This implies that for each $t>0$
\begin{align*}
\int_Xg_{E_o f}^{\rightarrow}(\cdot,t)^2\d\m&=\int_X|\sqrt{-\Delta}Q_t^{(0)}E_of|^2\d\m=\mathscr{E}(Q_t^{(0)}\!E_of,Q_t^{(0)}\!E_of)=0,\\
\int_Xg_{E_o f}^{\uparrow}(\cdot,t)^2\d\m&=\int_X\Gamma(Q_t^{(0)}\!E_of)\d\m=\mathscr{E}(Q_t^{(0)}\!E_of,Q_t^{(0)}\!E_of)=0,
\end{align*}
which yields the second conclusion.  
\end{proof} 
\subsection{Proof of \eqref{eq:LittlewoodPaleyStein1} 
for general $\alpha\geq C_{\kappa}$ with $\alpha>0$ and 
$p\in]1,+\infty[$}
\begin{proof}[\bf Proof of \boldmath\eqref{eq:LittlewoodPaleyStein1}]
It is clear that Subsections~\ref{subsec:3.2}, \ref{subsec:p>2} and \ref{subsec:p=2}
conclude the desired upper estimate \eqref{eq:LittlewoodPaleyStein1} for $f\in D(\Delta)\cap L^p(X;\m)$ and $p\in]1,+\infty[$. Any $f\in L^p(X;\m)\cap L^2(X;\m)$ can be approximated by $f_n:=P_{1/n}f\in D(\Delta)\cap L^p(X;\m)$ in $L^p$-norm and in $L^2$-norm. We can obtain the upper estimate \eqref{eq:LittlewoodPaleyStein1} for 
$f\in L^p(X;\m)\cap L^2(X;\m)$ under $p\in]1,+\infty[$ and $\alpha\geq C_{\kappa}$ with $\alpha>0$. 
Moreover, an $L^p$-approximation by a sequence in $L^p(X;\m)\cap L^2(X;\m)$ tells us that \eqref{eq:LittlewoodPaleyStein1} 
holds for $f\in L^p(X;\m)$ under $p\in]1,+\infty[$ and $\alpha\geq C_{\kappa}$ with $\alpha>0$. 
When $\alpha=0$ with $\kappa^-=0$, \eqref{eq:LittlewoodPaleyStein1} still holds for 
$f\in D(\Delta)\cap L^p(X;\m)$ under $p\in]1,+\infty[$. Then, similarly as above we get the same conclusion for $f\in L^p(X;\m)$ under $p\in]1,+\infty[$ and $\alpha=0$ with $\kappa^-=0$.  
\end{proof}

\subsection{Proof of \eqref{eq:LittlewoodPaleyStein2} for general $\alpha\geq C_{\kappa}$ with $\alpha>0$ and 
$p\in]1,+\infty[$}
\begin{proof}[\bf Proof of {\boldmath\eqref{eq:LittlewoodPaleyStein2}}]
The proof of the lower estimate \eqref{eq:LittlewoodPaleyStein2} follows from the upper estimate 
\eqref{eq:LittlewoodPaleyStein1}, 
\eqref{eq:EqE+1} and the 
argument in \cite[Subsection~3.2.11]{ShigekawaText}. 
Moreover, an $L^p$-approximation by a sequence in $L^p(X;\m)\cap L^2(X;\m)$ tells us that \eqref{eq:LittlewoodPaleyStein2} holds for 
$f\in L^p(X;\m)$ under $p\in]1,+\infty[$ and $\alpha\geq C_{\kappa}$ with $\alpha>0$. 
When $\alpha=0$ with $\kappa^-=0$, \eqref{eq:LittlewoodPaleyStein2} still holds for $f\in D(\Delta)\cap L^p(X;\m)$ under $p\in]1,+\infty[$ and the transience of $(\mathscr{E},D(\mathscr{E}))$ with Lemma~\ref{lem:SpectralMeasure}. 
Then, similarly as above we get the same conclusion for $f\in L^p(X;\m)$ under $p\in]1,+\infty[$, $\alpha=0$ with $\kappa^-=0$ and the transience of $(\mathscr{E},D(\mathscr{E}))$. 
\end{proof}

\subsection{Proofs of \eqref{eq:LittlewoodPaleyStein1+}, \eqref{eq:LittlewoodPaleyStein2+} and 
\eqref{eq:LittlewoodPaleyStein3+}
for  $\alpha=0$ and $p\in]1,+\infty[$ 
under $\kappa^-=0$}\label{subsec:LastProof}

We already obtain the estimates \eqref{eq:LittlewoodPaleyStein1+}, 
\eqref{eq:LittlewoodPaleyStein2+} and \eqref{eq:LittlewoodPaleyStein3+} under $p=2$. 
To prove  \eqref{eq:LittlewoodPaleyStein1+}, \eqref{eq:LittlewoodPaleyStein2+} and \eqref{eq:LittlewoodPaleyStein3+} for general $p\in]1,+\infty[$, we need to show $E_of\in L^p(X;\m)$ for 
$f\in  L^2(X;\m)$ at first. 
By Lemma~\ref{lem:SpectralMeasure}, we know $E_of\in D(\mathscr{E})$ and $\mathscr{E}(E_of,E_of)=0$.  When
$(\mathscr{E},D(\mathscr{E}))$ is transient, $E_of=0\in L^p(X;\m)$ in this case. 
When $(\mathscr{E},D(\mathscr{E}))$ is irreducible, then $E_of\equiv c_f\in L^2(X;\m)$. 
If $\m(X)=+\infty$, then $c_f=0$, hence $E_of=0\in L^p(X;\m)$. If $\m(X)<+\infty$, then $E_of=c_f\in L^p(X;\m)$. 

Under $\kappa^-=0$ and $\alpha=0$, we already know that \eqref{eq:LittlewoodPaleyStein1} and 
\eqref{eq:LittlewoodPaleyStein2} hold for $f\in L^p(X;\m)\cap L^2(X;\m)$ with $p\in]1,+\infty[$. Replacing $f$ with $f-E_of$ in \eqref{eq:LittlewoodPaleyStein1} and \eqref{eq:LittlewoodPaleyStein2} and noting $G_{E_of}=0$ by 
Lemma~\ref{lem:SpectralMeasure}, we can conclude that \eqref{eq:LittlewoodPaleyStein1+} and 
\eqref{eq:LittlewoodPaleyStein2+} 
 hold for $f\in L^p(X;\m)\cap L^2(X;\m)$ with $p\in]1,+\infty[$ under $\kappa^-=0$, $\alpha=0$ and the transience or irreducibility. Further, under $\kappa^-=0$ and $\alpha=0$, 
we can directly deduce the lower estimate \eqref{eq:LittlewoodPaleyStein3+} 
as discussed in \cite[Subsection~3.2.11]{ShigekawaText} 
from the upper estimate \eqref{eq:LittlewoodPaleyStein1+} by using \eqref{eq:EqEo2}.
This completes the whole proof of Theorem~\ref{thm:main1}.

\section{Proof of Theorem~\ref{thm:main2}}
The proof of Theorem~\ref{thm:main2} for $q>1$ is quite similar to the proof of 
\cite[Theorem~1.3]{KawabiMiyokawa} 
based on Theorem~\ref{thm:main1}. 
So we omit the details of the proof.
We only mention the well-definedness of $R_{\alpha}^{(q)}(\Delta)f$. Though 
the Riesz operator $f\mapsto R_{\alpha}^{(q)}(\Delta )f$ is not linear, we can confirm the following triangle inequality: for $f_1,f_2\in L^p(X;\m)\cap L^2(X;\m)$
\begin{align*}
\|R_{\alpha}^{(q)}(\Delta )f_1-R_{\alpha}^{(q)}(\Delta )f_2 \|_{L^p(X;\m)}^p&=
\int_X\left|\sqrt{\Gamma((\alpha-\Delta)^{-\frac{q}{2}}f_1)}-\sqrt{\Gamma((\alpha-\Delta)^{-\frac{q}{2}}f_2)} \right|^p\d\m\\
&\leq \int_X\Gamma((\alpha-\Delta)^{-\frac{q}{2}}(f_1-f_2))^{\frac{p}{2}}\d\m\\
&=\|R_{\alpha}^{(q)}(\Delta )(f_1-f_2) \|_{L^p(X;\m)}^p\\
&\leq \|R_{\alpha}^{(q)}(\Delta )\|_{p,p}\|f_1-f_2\|_{L^p(X;\m)},
\end{align*}
which implies the extendability of $R_{\alpha}^{(q)}(\Delta )f$ for $f\in L^p(X;\m)$.

\section{Examples}\label{sec:examples}
It is well-known that (abstract) Wiener space $(B,H,\mu)$ satisfies Littlewood-Paley-Stein inequality (see Shigekawa~\cite[Chapter~3]{ShigekawaText}). Though our Theorems~\ref{thm:main1} and \ref{thm:main2} are not new for $(B,H,\mu)$, we observe the conditions:   
Let $(\mathscr{E}^{\rm OU},D(\mathscr{E}^{\rm OU}))$ be the Dirichlet form on 
$L^2(B;\mu)$ associated to the Ornstein-Uhlenbeck process ${\bf X}^{\rm OU}$ and $(T_t^{\rm OU})_{t\geq0}$ its associated semigroup on $L^2(B;\mu)$. 
Let $D_H$ be the $H$-derivative, i.e. $\langle D_HF(z), h\rangle_H=\lim_{\eps\to0}\frac{F(z+\eps h)-F(z)}{\eps}$, for a cylindrical function $F$. It is known that  
${\sf BE}_2(1,\infty)$-condition holds for $(B,\mathscr{E}^{\rm OU},\mu)$ (see \cite[13.2]{AES}), that is, 
$(B,\mathscr{E}^{\rm OU},\mu)$ is tamed by $\mu\in S_K({\bf X}^{\rm OU})$. 

Moreover, the contents of Theorems~\ref{thm:main1} and \ref{thm:main2} are proved in Kawabi-Miyokawa~\cite{KawabiMiyokawa} under related different conditions (our Theorems~\ref{thm:main1} and \ref{thm:main2} do not cover the results of  \cite{KawabiMiyokawa}, but are not covered by \cite{KawabiMiyokawa}).  
More concretely, let $E$ be a Hilbert space defined by $E:=L^2(\R,\R^d;e^{-2\lambda\chi(x)}\d x)$ with a fixed $\chi\in C^{\infty}(\R)$ satisfying $\chi(x)=|x|$ for $|x|\geq1$ and another Hilbert space $H:=L^2(\R,\R^d;\d x)$.   
They consider a Dirichlet form $(\mathscr{E}, D(\mathscr{E}))$ on $L^2(E;\mu)$ associated with the diffusion process ${\bf X}$ on an infinite volume 
 path space $C(\R,\R^d)$ with ($U$-)Gibbs measures $\mu$ associated with the (formal) Hamiltonian
 \begin{align*}
 \mathcal{H}(w):=\frac12\int_{\R}|w'(x)|^2_{\R^d}\d x+\int_{\R}U(w(x))\d x, 
 \end{align*}  
 %where $U:\R^d\to\R$ is an interaction potential satisfying the conditions 
 where $U: \R^{d} \to \R$ is an interactions potential satisfying ${\nabla^2}U\geq -K_{1}$
with some $K_{1}\in\R$  
 % {\bf (U1)},{\bf (U2)},{\bf (U3)} depending on 
 %constants $K_1\in \R$ and 
 %$K_2>0$ 
 (see \cite[4.1]{KawabiMiyokawa}). 
 Then the $L^2$-semigroup $(P_t)_{t\geq0}$ associated to $(\mathscr{E}, D(\mathscr{E}))$ satisfies the following gradient estimate
\begin{align*}
|D(P_tf)(w)|_H\leq e^{K_1t}P_t(|Df|_H)(w)\quad \text{ for }\quad \mu\text{-a.e.~}w\in E,
\end{align*}
which is equivalent to ${\sf BE}_2(-K_1,\infty)$-condition. 
Here $Df$ is a closed extension of the Fr\'echet derivative $Df:E\to H$ for cylindrical function $f$.    
Hence $(E,\mathscr{E},\mu)$ is tamed by $-K_1\mu$ with $|K_1|\mu\in S_K({\bf X})$. 

\bigskip 

In the rest, we expose new examples. 
\begin{example}[{{\bf RCD spaces}}]
{\rm A metric measure space $(X,{\sf  d},\m)$ is a complete separable metric space $(X,{\sf d})$ with a $\sigma$-finite Borel measure with $\m(B)<\infty$ for any bounded Borel set $B$. We assume $\m$ has full topological support, i.e., $\m(G)>0$ for non-empty open set $G$. 

Any metric open ball is denoted by $B_r(x):=\{y\in X\mid {\sf d}(x,y)<r\}$ for $r>0$ and $x\in X$. 
A subset $B$ of $X$ is said to be bounded if it is included in a metric open ball.    
Denote by $C([0,1],X)$ the space of continuous curve defined on the unit interval $[0,1]$ equipped the distance ${\sf d}_{\infty}(\gamma,\eta):=\sup_{t\in[0,1]}{\sf d}(\gamma_t,\eta_t)$ for every $\gamma,\eta\in C([0,1],X)$. This turn $C([0,1],X)$ into complete separable metric space.   
Next we consider the set of $2$-absolutely continuous curves, denoted by $AC^q([0,1],X)$, is the subset of $\gamma\in C([0,1],X)$ so that there exists  $g\in L^q(0,1)$ satisfying 
\begin{align*}
{\sf d}(\gamma_t,\gamma_s)\leq\int_s^tg(r)\d r,\quad s<t\quad\text{ in }[0,1].
\end{align*}
Recall that for any $\gamma\in AC^2([0,1],X)$, there exists a minimal a.e.~function $g\in L^2(0,1)$ satisfying the above, called {\it metric speed} denoted by $|\dot\gamma_t|$, which is defined as 
$|\dot\gamma_t|:=\lim_{h\downarrow0}{\sf d}(\gamma_{t+h}, \gamma_t)/h$ for $\gamma\in AC^2([0,1],X)$, 
$|\dot\gamma_t|:=+\infty$ otherwise. 
We define the kinetic energy functional $C([0,1],X)\ni\gamma\mapsto {\sf Ke}_t(\gamma):=\int_0^1|\dot{\gamma}_t|^2\d t$, if $\gamma\in AC^2([0,1],X)$, ${\sf Ke}_t(\gamma):=+\infty$ otherwise. 
\begin{defn}[{{\bf {\boldmath$2$}-test plan}}]\label{def:$q$-test}
{\rm Let $(X,{\sf d},\m)$ be a metric measure space. 
A measure {\boldmath$\pi$}$\in\mathscr{P}(C([0,1],X))$ is said to be a $2$-test plan, provided 
\begin{enumerate}
\item[(i)]\label{item:qtest1} there exists $C>0$ so that $({\sf e}_t)_{\sharp}${\boldmath$\pi$}$\leq C\m$ for every $t\in[0,1]$;
\item[(ii)]\label{item:qtest2} we have $\int_{C([0,1],X)}{\sf Ke}_2(\gamma)${\boldmath$\pi$}$(\d\gamma)<\infty$. 
\end{enumerate} 
}
\end{defn}

\begin{defn}[{{\bf Sobolev space \boldmath{$W^{\hspace{0.03cm}1,2}(X)$}}}]\label{def:W1pSobolev}
{\rm
A Borel function $f\in L^2(X;\m)$ belongs to $W^{1,2}(X)$, 
provided there exists a $G\in L^2(X;\m)$,  
called {\it $2$-weak upper gradient} of $f$ so that 
\begin{align}
\int_{C([0,1],X)}|f(\gamma_1)-f(\gamma_0)|\text{\boldmath$\pi$}(\d\gamma)\leq \int_{C([0,1],X)}\int_0^1
G(\gamma_t)|\dot{\gamma}_t|\d t\text{\boldmath$\pi$}(\d\gamma),\quad \text{$\forall${\boldmath$\pi$} 
$2$-test plan}. \label{eq:SpSobolev}
\end{align} 
The assignment $(t,\gamma)\mapsto G(\gamma_t)|\dot{\gamma}_t|$ is Borel measurable (see \cite[Remark~2.1.9]{GPLecture}) and the right hand side of 
\eqref{eq:SpSobolev} is finite for $G\in L^2(X;\m)_+$ (see \cite[(2.5)]{GigliNobili}). These shows not only the finiteness of the right hand side of \eqref{eq:SpSobolev} but also the continuity of the assignment 
$L^2(X;\m)\ni G\mapsto  \int_{C([0,1],X)}\int_0^1 G(\gamma_t)|\dot{\gamma}_t|\d t\text{\boldmath$\pi$}(\d\gamma)$. This, combined with the closedness of the convex combination of the $2$-weak upper gradient, shows that the set of $2$-weak upper gradient of a given Borel function $f$ is a closed convex subset of $L^2(X;\m)$. The minimal $p$-weak upper gradient, denoted by $|Df|_2$ is then the element of 
minimal $L^p$-norm in this class. Also, by making use of the lattice property of the set of $2$-weak 
upper gradient, such a minimality is also in the $\m$-a.e.~sense (see \cite[Proposition~2.17 and Theorem~2.18]{Ch:metmeas}). 

Then, $W^{1,2}(X)$ forms a Banach space equipped with the following norm:
equipped with the norm 
\begin{align*}
\|f\|_{W^{1,2}(X)}:=\left(\|f\|^2_{L^2(X;\m)}+\| |Df|_2\|^2_{L^2(X;\m)} \right)^{\frac{1}{2}} ,\quad f\in  W^{1,2}(X). 
\end{align*}
 
}
\end{defn}

It is in general false 
that  $(W^{1,2}(X),\|\cdot\|_{W^{1,2}(X)})$ is a Hilbert space. 
When this occurs, we say that $(X,{\sf d},\m)$ is {\it infinitesimally Hilbertian} (see 
\cite{Gigli:OntheDifferentialStr}). Equivalently, we call $(X,{\sf d},\m)$  infinitesimally Hilbertian 
provided the following {\it parallelogram identity} holds:
\begin{align}
2|Df|_2^2+2|Dg|_2^2=|D(f+g)|_2^2+|D(f-g)|_2^2, \quad \m\text{-a.e.} \quad \forall f,g\in W^{1,2}(X).\label{eq:parallelogram}
\end{align}

For simplicity, when $p=2$, we omit the suffix $2$ from $|Df|_2$ for $f\in W^{1,2}(X)$, i.e. we write $|Df|$ instead of $|Df|_2$ for $f\in W^{1,2}(X)$.  
Under \eqref{eq:parallelogram}, we can give a bilinear form $\langle D\cdot ,D\cdot\rangle:W^{1,2}(X)\times W^{1,2}(X)\to L^1(X;\m)$ which is defined by 
\begin{align*}
\langle Df,Dg\rangle:=\frac14|D(f+g)|^2-\frac14|D(f-g)|^2,\quad f,g\in W^{1,2}(X).
\end{align*}
Moreover, under the infinitesimally Hilbertian condition, 
the bilinear form $(\mathscr{E},D(\mathscr{E}))$ defined by 
\begin{align*}
D(\mathscr{E}):=W^{1,2}(X),\quad \mathscr{E}(f,g):=\frac12\int_X\langle Df,Dg\rangle\d \m
\end{align*}
is a strongly local Dirichlet form on $L^2(X;\m)$. 
Denote by $(P_t)_{t\geq0}$ the $\m$-symmetric semigroup on $L^2(X;\m)$ associated with $(\mathscr{E},D(\mathscr{E}))$. 
Under \eqref{eq:parallelogram}, let $(\Delta, D(\Delta))$ be the 
$L^2$-generator associated with $(\mathscr{E},D(\mathscr{E}))$ similarly defined as in \eqref{eq:generatorL2} before. 
\begin{defn}[{{\bf RCD-spaces}}]
{\rm A metric measure space $(X,{\sf d},\m)$ is said to be an {\it {\sf RCD}$(K,\infty)$-space} if 
it satisfies 
the following conditions: 
\begin{enumerate}
\item[\rm(1)]
$(X, {\sf d}, \m)$ is infinitesimally Hilbertian. 

\item[\rm(2)]
There exist $x_0 \in X$ and constants $c, C > 0$ such that 
$\m(B_r(x_0)) \le C \e^{c r^2}$. 

\item[\rm(3)]
If $f \in W^{1,2}(X)$ satisfies 
$| D f |_2 \le 1$ $\m$-a.e., then $f$ has a $1$-Lipschitz representative. 

\item[\rm(4)]
For any $f \in D ( \Delta )$ 
with $\Delta f \in W^{1,2}(X)$ 
and $g \in D ( \Delta ) \cap L^\infty (X; \m)$ 
with $g \ge 0$ and $\Delta g \in L^\infty (X; \m)$, 
\begin{align*}
\frac12 \int_X | D f |^2 \Delta g \, \d \m 
- \int_X \langle D f, D \Delta f \rangle g \, \d \m 
\ge 
K \int_X | D f |^2 g \, \d \m. 
\end{align*}
\end{enumerate}
Let $N\in[1,+\infty[$. 
A metric measure space $(X,{\sf d},\m)$ is said to be an {\it {\sf RCD}$(K,N)$-space} if 
it is an {\sf RCD}$(K,\infty)$-space and 
for any $f \in D ( \Delta )$ 
with $\Delta f \in W^{1,2}(X)$ 
and $g \in D( \Delta ) \cap L^\infty (X; \m)$ 
with $g \ge 0$ and $\Delta g \in L^\infty (X; \m)$, 
\begin{align*}
\frac12 \int_X | D f |^2 \Delta g \, \d \m 
- \int_X \langle D f, D \Delta f \rangle g \, \d \m 
\ge 
K \int_X | D f |^2 g \, \d \m 
+ \frac{1}{N} \int_X ( \Delta f )^2 g \, \d \m. 
\end{align*}

}
\end{defn}
\begin{remark}
{\rm \quad
\begin{enumerate}
\item It is shown in \cite{Cav-Mil,ZLi} that for 
$N\in[1,+\infty[$ {\sf RCD}$(K,N)$-space is equivalent to 
{\sf RCD}${}^*(K,N)$-space, where {\sf RCD}${}^*(K,N)$-space  %(resp.~{\sf RCD}${}^*(K,N)$-space) 
is defined to be a {\sf CD}${}^*(K,N)$-space 
%(resp.~{\sf CD}${}^*(K,N)$-space)
 having infinitesimal Hilbertian condition. 
%Here {\sf CD}${}^*(K,N)$-space is a metric measure space defined in terms of optimal mass transport theory (see \cite{AGS_Riem,EKS} for details). 
Originally, for $N\in[1,+\infty]$, the notion of {\sf CD}$(K,N)$-spaces was defined by Lott-Villani~\cite{LV2} and Sturm~\cite{StI,StII,StICorrect}. Later, Bacher-Sturm~\cite{Bacher-Sturm} gave the notion of 
{\sf CD}${}^*(K,N)$-space as a variant for $N\in[1,+\infty[$. Finally, Erbar-Kuwada-Sturm~\cite{EKS} invented the notion of 
{\sf CD}${}^e(K,N)$-space, so-called the metric measure space satisfying entropic curvature dimension condition under $N\in[1,+\infty[$ and prove that {\sf RCD}${}^*(K,N)$-space coincides with {\sf RCD}${}^e(K,N)$-space, i.e., 
{\sf CD}${}^e(K,N)$-space satisfying infinitesimal Hilbertian condition. 
The notion {\sf RCD}$(K,\infty)$-space, i.e., {\sf CD}$(K,\infty)$-space satisfying 
infinitesimal Hilbertian condition, was firstly considered in \cite{AGS_Riem,AGS_BakryEmery}, and 
the {\sf RCD}$(K,N)$-space under $N\in[1,+\infty[$ was also given by \cite{EKS,Gigli:OntheDifferentialStr}.
%So we may say  
%{\sf RCD}$(K,N)$-space instead of {\sf RCD}${}(K,N)$-space. 
Moreover, under $N\in[1,+\infty[$, {\sf RCD}$(K,N)$-space (or {\sf RCD}${}^*(K,N)$-space) is a locally compact separable metric space, consequently, $\m$ becomes a Radon measure.  
\item If $(X,{\sf d},\m)$ is an {\sf RCD}$(K,N)$-space under $N\in[1,+\infty[$, 
it enjoys the Bishop-Gromov inequality: 
Let $\kappa:=K/(N-1)$ if $N>1$ and $\kappa:=0$ if $N=1$. We set $\omega_N:=\frac{\pi^{N/2}}{\int_0^{\infty}t^{N/2}e^{-t}\d t}$ (volume of unit bal in $\R^N$ provided $N\in\mathbb{N}$) and 
$V_{\kappa}(r):=\omega_N\int_0^r\s_{\kappa}^{N-1}(t)\d t$.  
 Then 
\begin{align}
\frac{\m(B_R(x))}{V_{\kappa}(R)}\leq \frac{\m(B_r(x))}{V_{\kappa}(r)},\qquad x\in X, \qquad 0<r<R. \label{eq:BishopGromov}
\end{align}
Here  
$\s_{\kappa}(s)$ is the solution to Jacobi equation $\s_{\kappa}''(s)+\kappa\s_{\kappa}(s)=0$ with 
$\s_{\kappa}(0)=0$, $\s_{\kappa}'(0)=1$. More concretely, $\s_{\kappa}(s)$ is given by
\begin{align*}
\s_{\kappa}(s):=\left\{\begin{array}{cc}\frac{\sin \sqrt{\kappa}s}{\sqrt{\kappa}} & \kappa>0, \\ s & \kappa=0, \\ \frac{\sinh \sqrt{-\kappa}s}{\sqrt{-\kappa}} & \kappa<0.\end{array}\right.
\end{align*}
\item If $(X,{\sf d},\m)$ is an $\mathsf{RCD}(K,\infty)$-space, it satisfies the following Bakry-\'Emery estimate:
\begin{align}
|DP_tf|\leq e^{-Kt}P_t|Df|\quad \m\text{-a.e.}\quad for \quad f\in W^{1,2}(X),\label{eq:BakryEmery}
\end{align}
in particular, $P_tf\in W^{1,2}(X)$
(see \cite[Corollary~3.5]{Sav14}, \cite[Proposition~3.1]{GigliHan}).
\item 
If $(X,{\sf d},\m)$ is an $\mathsf{RCD}(K,N)$-space with $N\in[1,+\infty[$, then,  
for any $f\in {\rm Lip}(X)$ 
${\rm lip}(f)=|Df|$ $\m$-a.e. holds, 
because  $\mathsf{RCD}(K,N)$ with $N\in[1,+\infty[$ admits the local volume doubling and a weak local $(1,2)$-Poincar\'e inequality (see \cite[\S5 and \S6]{Ch:metmeas}). 
\end{enumerate}
}
\end{remark}
By definition, for ${\sf RCD}(K,N)$-space $(X,{\sf d},\m)$,  $(X,\mathscr{E},\m)$ 
satisfies ${\sf BE}_2(K,N)$-condition and $K^{\pm}\m$ is always a Kato class smooth measure, 
hence it is a tamed Dirichlet space by $K\m$. So we can apply Theorems~\ref{thm:main1} and \ref{thm:main2} to $(X,{\sf d},\m)$. We remark that Theorem~\ref{thm:main1} is proved in \cite{HLiWeighted,HLi} for 
${\sf RCD}(K,\infty)$-space. However, the result in \cite{HLiWeighted,HLi} 
can be deduced from Kawabi-Miyokawa~\cite{KawabiMiyokawa}, because the space ${\rm Test}(X)$ 
of test functions forms an subalgebra of $C_b(X)$, and hence it satisfies the condition {\bf (A)} in  \cite{KawabiMiyokawa} 
(see Bouleau-Hirsch~\cite[Corollary~4.2.3]{BH} for the validity of {\bf (A)} in 
\cite{KawabiMiyokawa}).  
}
\end{example}

\begin{example}[{{\bf Riemannian manifolds with boundary}}]
{\rm \quad Let $(M,g)$ be a smooth Riemannian manifold with boundary $\partial M$. 
Denote by $\mathfrak{v}:={\rm vol}_g$ the Riemannian volume measure induced by $g$, and by 
$\mathfrak{s}$ the surface measure on $\partial M$
(see \cite[\S1.2]{Braun:Tamed2021}). 
If $\partial M\ne \emptyset$, then $\partial M$ is a smooth co-dimension $1$ submanifold of $M$ and it becomes Riemannian 
when endowed with the pullback metric 
\begin{align*}
\langle \cdot,\cdot\rangle_j:=j^*\langle\cdot,\cdot\rangle,\qquad \langle u,v\rangle:=g(u,v)\quad\text{ for }\quad u,v\in TM.
\end{align*}  
under the natural inclusion $j:\partial M\to M$. The map $j$ induces a natural inclusion $d_j:T\partial M\to TM|_{\partial M}$ which is not surjective. In particular, the vector bundles $T\partial M$ and $TM|_{\partial M}$ do not coincide. 

Let $\m$ be a Borel measure on $M$ which is locally equivalent to $\mathfrak{v}$. 
Let $D(\mathscr{E}):=W^{1,2}(M^{\circ})$ be the Sobolev space with respect to 
$\m$ defined in the usual sense on $M^{\circ}:=M\setminus\partial M$. Define $\mathscr{E}:W^{1,2}(M^{\circ})\to [0,+\infty[$ 
by 
\begin{align*}
\mathscr{E}(f):=\int_{M^{\circ}}|\nabla f|^2\d \m
\end{align*}  
and the quantity $\mathscr{E}(f,g)$, $f,g\in W^{1,2}(M^{\circ})$, by polarization. Then $(\mathscr{E}, D(\mathscr{E}))$ becomes a strongly local regular Dirichlet form on $L^2(M;\m)$, since $C_c^{\infty}(M)$ is a dense set of $D(\mathscr{E})$. 
Let $k:M^{\circ}\to\R$ and $\ell:\partial M\to\R$ be continuous functions providing lower bounds on the Ricci curvature and the second fundamental form of $\partial M$, respectively. 
Suppose that $M$ is compact and $\m=\mathfrak{v}$. Then $(M,\mathscr{E},\m)$ is tamed by 
\begin{align*}
\kappa:=k\mathfrak{v}+\ell\mathfrak{s},
\end{align*}
because $\mathfrak{v},\mathfrak{s}\in S_K({\bf X})$ (see \cite[Theorem~2.36]{ERST} and \cite[Theorem~5.1]{Hsu:2001}). Then one can apply Theorems~\ref{thm:main1} and \ref{thm:main2} to $(M,\mathscr{E},\m)$. 
Remark that \eqref{eq:LittlewoodPaleyStein1} is proved by 
Shigekawa~\cite[Propositions~6.2 and 6.4]{Shigekawa1} in the framework of compact smooth Riemannian manifold with boundary. 

More generally, if $M$ is regularly exhaustible, i.e., there exists an increasing sequence $(X_n)_{n\in\N}$ of domains $X_n\subset M^{\circ}$ with smooth boundary $\partial X_n$ such that $g$ is smooth on $X_n$ and the following properties hold:
\begin{enumerate}
\item[(1)] The closed sets $(\overline{X}_n)_{n\in\N}$ constitute an $\mathscr{E}$-nest for $(\mathscr{E},W^{1,2}(M))$.  
\item[(2)] For all compact sets $K\subset M^{\circ}$ there exists $N\in\N$ such tht $K\subset X_n$ for all $n\geq N$.
\item[(3)] There are lower bounds $\ell_n:\partial X_n\to\R$ for the curvature of $\partial X_n$ with $\ell_n=\ell$ on $\partial M\cap \partial X_n$ such that the distributions $\kappa_n=k\mathfrak{v}_n+\ell_n\mathfrak{s}_n$ are uniformly $2$-moderate in 
the sense that 
\begin{align}
\sup_{n\in\N}\sup_{t\in[0,1]}\sup_{x\in X_n}\E_x^{(n)}\left[e^{-2A_t^{\kappa_n}} \right]<\infty,\label{eq:2moderate}
\end{align}
where $\mathfrak{v}_n$ is the volume measure of $X_n$ and $\mathfrak{s}_n$ is the surface measure of $\partial X_n$.
\end{enumerate}
Suppose $\m=\mathfrak{v}$. 
Then $(M,\mathscr{E},\m)$ is tamed by $\kappa=k\mathfrak{v}+\ell\mathfrak{s}$ 
(see \cite[Theorem~4.5]{ERST}), hence Theorems~\ref{thm:main1} and \ref{thm:main2} hold for $(M,\mathscr{E},\m)$ provided 
$\kappa^+\in S_D({\bf X})$ and $2\kappa^-\in S_{E\!K}({\bf X})$. 

Let $Y$ be the domain defined by 
\begin{align*}
Y:=\{(x,y,z)\in\R^3\mid z>\phi(\sqrt{x^2+y^2})\},
\end{align*}
where $\phi:[0,+\infty[\to[0,+\infty[$ is $C^2$ on $]0,+\infty[$ with $\phi(r):=r-r^{2-\alpha}$, $\alpha\in]0,1[$ for 
$r\in[0,1]$, $\phi$ constant for $r\geq2$ and $\phi''(r)\leq0$ for $r\in[0,+\infty[$. 
Let $\m_Y$ be the $3$-dimensional Lebesgue measure restricted to $Y$ and 
$\sigma_{\partial Y}$ the $2$-dimensional Hausdorff measure on $\partial Y$. Denote by 
$\mathscr{E}_Y$ the Dirichlet form on $L^2(Y;\m)$ with Neumann boundary conditions. 
The smallest eigenvalue of the second fundamental form of $\partial Y$ can be given by 
\begin{align*}
\ell(r,\phi)=\frac{\phi''(r)}{(1+|\phi'(r)|^2)^{3/2}} (\leq0),
\end{align*}
for $r\leq 1$ and $\ell=0$ for $r\geq2$. It is proved in \cite[Theorem~4.6]{ERST} that 
the Dirichlet space $(Y,\mathscr{E}_Y,\m_Y)$ is tamed by 
\begin{align*}
\kappa=\ell\sigma_{\partial Y}.
\end{align*}
Since $|\kappa|=|\ell|\sigma_{\partial Y}\in S_K({\bf X})$ (see \cite[Lemma~2.34, Theorem~2.36, Proof of Theorem~4.6]{ERST}), 
we can apply Theorems~\ref{thm:main1} and \ref{thm:main2} for $(Y,\mathscr{E}_Y,\m_Y)$. 
}
\end{example}
\begin{example}[{{\bf Configuration space over metric measure spaces}}]
{\rm Let $(M,g)$ be a complete smooth Riemannian manifold without boundary.   
The configuration space $\Upsilon$ over $M$ is the space of all locally finite point measures, that is, 
\begin{align*}
\Upsilon:=\{\gamma\in\mathcal{M}(M)\mid \gamma(K)\in\N\cup\{0\}\quad \text{ for all compact sets}\quad K\subset M\}. 
\end{align*}
In the seminal paper Albeverio-Kondrachev-R\"ockner~\cite{AKR} identified a natural geometry on $\Upsilon$ by lifting the geometry of $M$ to $\Upsilon$. In particular, there exists a natural gradient $\nabla^{\Upsilon}$, divergence ${\rm div}^{\Upsilon}$ and 
Laplace operator $\Delta^{\Upsilon}$ on $\Upsilon$. It is shown in \cite{AKR} that the Poisson point measure $\pi$ on 
$\Upsilon$ is the unique (up to intensity) measure on $\Upsilon$ under which the gradient and divergence become 
dual operator in $L^2(\Upsilon;\pi)$. Hence, the Poisson measure $\pi$ is the natural volume measure on $\Upsilon$ 
and $\Upsilon$ can be seen as an infinite dimensional Riemannian manifold. The canonical Dirichlet form 
\begin{align*}
\mathscr{E}(F)=\int_{\Upsilon}|\nabla^{\Upsilon}F|_{\gamma}^2\pi(\d\gamma)
\end{align*}
 constructed in \cite[Theorem~6.1]{AKR} is quasi-regular and strongly local
 and it induces the heat semigroup $T_t^{\Upsilon}$ and a Brownian motion ${\bf X}^{\Upsilon}$ on $\Upsilon$ which can be identified with the independent infinite particle process. If ${\rm Ric}_g\geq K$ on $M$ with $K\in \R$, then $(\Upsilon,\mathscr{E}^{\Upsilon},\pi)$ is tamed by 
 $\kappa:=K\pi$ with $|\kappa|\in S_K({\bf X}^{\Upsilon})$ (see \cite[Theorem~4.7]{EKS} and \cite[Theorem~3.6]{ERST}). 
 Then one can apply Theorems~\ref{thm:main1} and \ref{thm:main2} for $(\Upsilon,\mathscr{E}^{\Upsilon},\pi)$. 
 
 More generally, in Dello Schiavo-Suzuki~\cite{DelloSuzuki:ConfigurationI}, configuration space $\Upsilon$ over proper complete and separable metric space $(X,{\sf d})$ is considered. The configuration space $\Upsilon$ is endowed with the \emph{vague topology} $\tau_V$, induced by duality with continuous compactly supported functions on $X$, and with a reference Borel probability measure $\mu$ satisfying \cite[Assumption~2.17]{DelloSuzuki:ConfigurationI}, commonly understood as the law of a proper point process on $X$. In \cite{DelloSuzuki:ConfigurationI}, 
 they constructed the strongly local Dirichlet form $\mathscr{E}^{\Upsilon}$ defined to be the $L^2(\Upsilon;\mu)$-closure of a certain pre-Dirichlet form on a class of certain cylinder functions and prove its quasi-regularity for a wide class of measures $\mu$ and base spaces (see \cite[Proposition~3.9 and Theorem~3.45]{DelloSuzuki:ConfigurationI}). Moreover, in 
 Dello Schiavo-Suzuki~\cite{DelloSuzuki:ConfigurationII}, for any fixed $K\in \R$ they prove that a Dirichlet form $(\mathscr{E},D(\mathscr{E}))$ with its 
 carr\'e-du-champ $\Gamma$ satisfies ${\sf BE}_2(K,\infty)$ if and only if the Dirichlet form $(\mathscr{E}^{\Upsilon},D(\mathscr{E}^{\Upsilon}))$ on $L^2(\Upsilon;\mu)$ with its carr\'e-du-champ $\Gamma^{\Upsilon}$ satisfies ${\sf BE}_2(K,\infty)$.
Hence, if $(X,\mathscr{E},\m)$ is tamed by $K\m$ with $|K|\m\in S_K({\bf X})$, then  $(\Upsilon,\mathscr{E}^{\Upsilon},\mu)$ is tamed by $\kappa:=K\mu$ with $|\kappa|\in S_K({\bf X}^{\Upsilon})$. 
Then one can apply Theorems~\ref{thm:main1} and \ref{thm:main2} for $(\Upsilon,\mathscr{E}^{\Upsilon},\mu)$ under the suitable class of measures $\mu$ defined in \cite[Assumption~2.17]{DelloSuzuki:ConfigurationI}.  
}
\end{example}

\noindent
{\bf Conflict of interest.} The authors have no conflicts of interest to declare that are relevant to the content of this article.

\bigskip

\noindent
{\bf Data Availability Statement.} Data sharing is not applicable to this article as no datasets were generated or analyzed during the current study.

\bigskip

\noindent
{\bf Acknowledgment.} 
The authors would like to thank Professors Mathias Braun and Luca Tamanini for telling us 
the references \cite{ZLi,HLi} after the first draft of this paper. 

\bigskip

%The authors would like to tell 
%their sincere gratitude to the anonymous referee. 
%His/Her comments help to improve the quality 
%of this paper very much.  

\providecommand{\bysame}{\leavevmode\hbox to3em{\hrulefill}\thinspace}
\providecommand{\MR}{\relax\ifhmode\unskip\space\fi MR }
\providecommand{\MRhref}[2]{%
  \href{http://www.ams.org/mathscinet-getitem?mr=#1}{#2}
}
\providecommand{\href}[2]{#2}

% \bibliographystyle{amsplain}
% \bibliography{refs}
\end{document}